\documentclass[reqno, 11 pt]{amsart}%
\usepackage{amsthm}
\usepackage{amsmath}
\usepackage{amsfonts}
\usepackage{txfonts}
\usepackage{amssymb}
\usepackage{dsfont}
\usepackage{esint}
\usepackage{dsfont}
\usepackage{graphicx}
\usepackage{color}
\usepackage{pdfrender,xcolor}
\usepackage{lmodern}
\usepackage{setspace}
\usepackage{bm}
\usepackage{url}
\usepackage[utf8]{inputenc}
\usepackage[T1]{fontenc}
\usepackage[english]{babel}
\usepackage{pdfrender,xcolor}
\usepackage{lmodern}
\usepackage{bm}
\usepackage{comment}
\usepackage[margin=0.8in]{geometry}
\usepackage{enumerate}
\usepackage{enumitem}
\usepackage{amsaddr}

\usepackage{empheq}%
\usepackage{mathtools}
\providecommand{\U}[1]{\protect\rule{.1in}{.1in}}
\numberwithin{equation}{section}

\newcommand{\pt}{\partial_t}
\newcommand{\bu}{\boldsymbol{u}}
\newcommand{\ba}{\boldsymbol{a}}

\newcommand{\bc}{\boldsymbol{c}}

\newcommand{\bv}{\boldsymbol{v}}

\newcommand{\bg}{\boldsymbol{g}}
\newcommand{\bD}{\textbf{D}}
\newcommand{\bH}{\textbf{H}}

\usepackage{mathrsfs}
\usepackage{amsbsy}

\newcommand{\Dujl}{\textbf{D}(\boldsymbol{u}^{j,l})}
\newcommand{\Dujlt}{\textbf{D}(\boldsymbol{u}^{j,l}(t))}

\newcommand{\Duj}{\textbf{D}(\boldsymbol{u}^j)}

\newcommand\nut{\nu_{\operatorname{turb}}}

\newtheorem{theorem}{Theorem}

\newtheorem{lemma}{Lemma}
\newtheorem{proposition}{Proposition}

\newtheorem{remark}{Remark}

\newcommand\R{\mathds{R}}
\allowdisplaybreaks

\newcommand\restr[2]{{
  \left.\kern-\nulldelimiterspace 
  #1 
  \vphantom{\big|} 
  \right|_{#2} 
  }}

\usepackage{mathrsfs,hyperref}

\usepackage[sort,nocompress]{cite} 

\let\origmaketitle\maketitle
\def\maketitle{
  \begingroup
  \def\uppercasenonmath##1{} 
  \let\MakeUppercase\relax 
  \origmaketitle
  \endgroup
}

\begin{document}

\inputencoding{utf8}

\title[Existence with unbounded turbulent-depending coefficients]{Existence for turbulent flows through permeable media with unbounded turbulent-depending coefficients}

\author[H.B.~de~Oliveira]{Hermenegildo Borges de Oliveira$^{(1, 2)}$}

\begingroup
\let\center\flushleft
\let\endcenter\endflushleft
\address{$^{(1)}$FCT - Universidade do Algarve, Faro, Portugal\\
         $^{(2)}${{CMAFcIO}} - Universidade de Lisboa, Lisboa, Portugal}
\endgroup

\email{holivei@ualg.pt}
\date{\today}

\selectlanguage{english}

\begin{abstract}
\noindent
A mathematical model that governs turbulent flows through permeable media is considered in this work.
The model under consideration is based on a double-averaging concept which in turn is described by the time-averaging technique characteristic of the turbulence k$-$epsilon  model and by the volume-averaging methodology that is used to model unstable flows through porous media.
The functions of turbulence viscosity, turbulence diffusion and turbulence production are assumed to be unbounded with respect to the turbulent kinetic energy.
For the associated initial-and boundary-value problem, we prove the existence of suitable weak solutions.

\bigskip


\bigskip

\end{abstract}

\maketitle

\section{Introduction}\label{Sect-Int}

In this work we study turbulent flows in a permeable medium with a hydraulic diameter large enough so that the fluid can be considered in the turbulent regime.
The most used approach to study turbulence in permeable media is based on $k-$epsilon modeling.
By this approach, macroscopic turbulence models for incompressible single-phase flow in rigid and fully saturated permeable media are derived using two distinct averaging concepts.
The Reynolds-averaged Navier-Stokes (RANS) equations are first developed at the microscale by time-averaging the incompressible Navier-Stokes equations.
Then, by volume-averaging the RANS equations, we obtain a macroscale equation for the evolution of turbulence.
The total drag effect due to the skeleton of the permeable medium is modeled only after applying the two average concepts.
Proceeding in this way, we obtain the same set of equations regardless of the order in which the two average concepts are applied~\cite{PL:2000,Vafai:2005,WHA:2020}.
In light of this double decomposition approach, we consider in this work the following general one-equation turbulence model,
\begin{eqnarray}
&& \partial_t{\bu}+({\bu}\cdot\nabla){\bu}- \mathbf{div}\left(\nut(k)\mathbf{D}({\bu})\right) + \nabla p=\mathbf{g}-c_{Da}{\bu}-c_{Fo}|{\bu}|^{\alpha-2}{\bu}\qquad \mbox{in}\ \ Q_T, \label{problem1}\\
&& \operatorname{div}{\bu}=0\qquad \mbox{in}\ \ Q_T,\label{problem2}\\
&& \partial_tk+{\bu}\cdot\nabla k -\operatorname{div}(\nu_D(k)\nabla k)=\nut(k)|\mathbf{D}({\bu})|^2+\nu_P(k)|{\bu}|^\beta-\varepsilon(k)\qquad \mbox{in}\ \ Q_T,\label{problem3}\\
&& {\bu}={\bu}_0\quad\mbox{and}\quad k=k_0\qquad \mbox{in}\ \ \Omega\times\{0\}, \label{problem4}\\
&& {\bu} =\mathbf{0}\quad\mbox{and}\quad k=0\qquad \mbox{on}\ \ \Gamma_T,  \label{problem5}
\end{eqnarray}
where $Q_T:=\Omega\times(0,T)$ is a space-time cylinder with lateral boundary $\Gamma_T:=\partial\Omega\times(0,T)$, being $\Omega\subset\mathds{R}^d$ a bounded  domain (open and connected) with its boundary denoted by $\partial\Omega$, and $T$ is a given positive constant.
Despite real world problems correspond to $d=3$, and in certain particular cases $d=2$, we consider a general space dimension $d$, to be restricted later on.
In (\ref{problem1})-(\ref{problem5}), the velocity field ${\bu}$ and the pressure $p$ are, in fact, averages that result by the application of the two aforementioned averaging concepts~\cite{Lemos:book}. %
The  averaged tensor $\mathbf{D}({\bu})$ is the symmetric part of the averaged gradient $\nabla {\bu}$.
For the sake of simplifying the problem, we assume the porosity of the medium is constant which justifies writing the mean flow equation in the form (\ref{problem1}).
The symbol $\mathbf{g}$ on the r.h.s. of the mean flow equation (\ref{problem1}) stay in this work for a general (averaged) body force, for instance the gravity force.
In the same equation, the feedback terms
$$c_{Da}{\bu}+c_{Fo}|{\bu}|^{\alpha-2}{\bu}$$
account for the resistance made by the skeleton of the permeable medium to the flow.
Here, $c_{Da}$ and $c_{Fo}$ are positive constant that are experimentally determined, usually denoted as the Darcy and Forchheimer coefficients.
The exponent $\alpha$ ranges in the interval $(1,\infty)$ and is a constant that characterizes the flow.
In particular, when $\alpha=2$ we obtain solely the Darcy term which accounts for the viscous drag, and if $\alpha=3$, we obtain the superposition of the Darcy and Forchheimer terms that account for both the viscous and form drags.
The function $k$ is an unknown of the problem and is usually called turbulent kinetic energy (TKE).
By definition, we always have
\begin{equation*}
k\geq 0.
\end{equation*}
The scalar function $\nut$ is the turbulent, or eddy viscosity, that may depend on $k$, whereas $\nu_D$ is the turbulent diffusion that may also depend on $k$.
The function $\varepsilon$ describes the rate of dissipation of the TKE in the model and therefore it is denoted by
dissipation of the TKE, or, briefly, turbulent dissipation.
In standard models,
\begin{equation*}
\varepsilon(k) = \frac{k \sqrt { | k |}}{\ell}
\end{equation*}
where $\ell: Q_T\longrightarrow\mathds{R}$ is the Prandtl length scale (function) of the motion, which is usually assumed to satisfy $\ell\geq \ell_0$ a.e. in $Q_T$ for some positive constant $\ell_0$.
Therefore, without loss of generality, we can assume that
\begin{equation}\label{epslion(k):e(k)}
\varepsilon(k)=k\,e(k),
\end{equation}
with
\begin{equation}\label{f:diss-turb2}
e(k)\geq 0\quad \forall\ k \in \mathds{R}_0^+,\quad \mbox{a.e. in}\ Q_T.
\end{equation}
In particular,
\begin{equation*}
\varepsilon(k) k  \geq 0\quad \forall \, k \in \R_0^+,  \quad \mbox{a.e. in}\ Q_T.
\end{equation*}

The additional term $\nu_P(k)|{\bu}|^\beta$ in equation (\ref{problem3}) appears as an output of the averaging process, and it is a production term of turbulent kinetic energy that accounts for the solids inside the fluid. Therefore $\nu_P$ shall be called the turbulence production function.
Several expressions for the function $\nu_P$ and for the exponent $\beta$ have been considered in the applications.
In particular, for $\nu_P(k)=k$ and $\beta=1$, we recover the turbulence model~\cite{PL:2000}, and for $\nu_P(k)$ constant and $\beta=3$ we get the turbulence model \cite{NK:1999}.

Problem (\ref{problem1})-(\ref{problem5}) can be easily adapted to cover other turbulence modeling situations not directly related to permeable media~\cite{OP:2018a,O:2024a,O:2023a}.
In particular, considering zero drag forces and no turbulence production term, and assuming that the turbulent dissipation $\varepsilon(k)$ is on the order of $k^{\frac{3}{2}}$, we recover the one-equation turbulence $k$--epsilon model~\cite{CL:2014,MP:1993}.
The mathematical analysis of this model has been investigated during the last 20-30 years, although important questions, such as the case of real turbulent viscosity and turbulent diffusion functions, remain open.
In clear flow conditions, that is for turbulent flows with zero drag forces and without the producing turbulence term, we address the reader to the works~\cite{BLM:2011,D:2008,GLLMT:2003,LL:2007,L:1997,MN:2015} for questions of existence, uniqueness and regularity of the solutions.
The turbulent model studied in the present work differs from the models studied in these references in two essential aspects.
The first lies in the presence of the viscous and form drag terms, $c_{Da}{\bu}$ and $c_{Fo}|{\bu}|^{\alpha-2}{\bu}$, in the mean flow equation (\ref{problem1}).
The second results from the fact that these two terms induce the production of more turbulence, which is described in the model by the extra non-linear term $\nu_P(k)|\bu|^\beta$ in equation (\ref{problem3}).
To the best of our knowledge, the mathematical analysis of the problem (\ref{problem1})-(\ref{problem5}) began in the works~\cite{OP:2016,OP:2017,OP:2018a,O:2018b,OP:2019} , where the authors studied issues regarding to the existence of solutions to the stationary version of the problem, as well as some aspects of the regularity of these solutions.
On the other hand, the effect of the generalized Forchheimer term $|\bu|^{\alpha-2}\bu$ on the incompressible Navier-Stokes equations (in the laminar regime) has been studied in~\cite{ADO:2002,ADO:2004a,ADO:2004b,ADO:2007}, in particular to obtain the confinement of the solutions, either in space~\cite{ADO:2002,ADO:2004a,ADO:2004b} or in time~\cite{ADO:2007}.
Very recently~\cite{O:2023a,O:2024a}, the existence of suitable weak solutions to the problem (\ref{problem1})-(\ref{problem5}) was proven, under the strong constraint that turbulence functions $\nut$ , $\nu_D$ and $\nu_P$ are bounded.
The present work improves the results established in~\cite{O:2024a,O:2023a} in the sense that we are now removing the restrictions on the boundedness of the turbulent functions $\nut$, $\nu_D$ and $\nu_P$


Our problem has some resemblances with the Navier-Stokes-Fourier system governing clear flows, in the laminar regime, of incompressible fluids with temperature-dependent coefficients~\cite{BFM:2009}.
Mathematically speaking, the main difficulty of these problems lies in the first r.h.s. term of the turbulence equation \eqref{problem3} (or energy equation for the Navier-Stokes-Fourier case), which is only in $L^1$, making that passing the approximate equation of the weak formulation to the limit does not preserve the identity.
To overcome the low regularity of that nonlinear term, the authors in~\cite{BFM:2009} considered the equation that results from adding the scalar product of the momentum equation and the velocity field with the energy equation, obtaining an extra equation for a new quantity that is expressed as the sum of the kinetic energy with the internal energy.
However, in this new equation, it is not possible to get rid of the pressure, as we can in the incompressible Navier-Stokes equations.
Thus, and as the applicability of de Rham's lemma to Navier-Stokes equations with variable coefficients is still unknown, the authors~\cite{BFM:2009} preferred to work with Navier's slip boundary conditions for the velocity field.
This, together with the assumption that the boundary is, at least, $C^{1,1}$, lead to the existence of globally integrable pressure.
Furthermore, the authors recovered an inequality of the type~\eqref{weak-form-k} (see below) by making use of the second law of thermodynamics.
By these approach the authors~\cite{BFM:2009} were able to prove the long-time and large-data existence of suitable weak solutions.
The same reasoning was used in~\cite{BLM:2011} to study a one-equation $k-$epsilon model governing turbulence in clear flows.

This paper is organized as follows. In this section (Section~\ref{Sect-Int}), we have introduced the problem we shall work with and gave the motivation of the real world situation.
The main result of this work  (Theorem~\ref{thm:exist}) is presented in Section ~\ref{Sect:main}.
From Section~\ref{Sect:trunc} till Section~\ref{Sect-Att-ic} we prove Proposition~\ref{prop:exist:n} which concerns the existence of suitable weak solutions for the truncated problem.
The proof of Theorem~\ref{thm:exist} is then concluded in Sections~\ref{Sect-Est-ind(n)} and~\ref{Sect-Att-ic}.
The notation used in this work is quite standard in the field of Mathematical Fluid Mechanics.
In any case, we address the interested reader to some of the monographs cited hereinafter~\cite{CL:2014,Galdi:2011,Temam:1979}.
We just want to point out that boldface letters denote tensor-valued (capital) and vector-valued (small) functions and non-boldface letters stay for scalars.
The letters $C$, $K$ and $\aleph$ will always denote positive constants, whose values may change from line to line, but whose dependence on other parameters or data will always be clear from the exposition.
We will only emphasize their dependence on the parameters that will later be passed to the limit.
Bellow, we recall the well-know notation for the function spaces considered in the analysis of incompressible viscous fluids,
\begin{eqnarray*}
& & \mathcal{V}:=\{\bv\in C_0^{\infty}(\Omega)^d:\mathrm{div}\bv=0\} \\
& & \mathbf{H}:=\mbox{closure of $\mathcal{V}$ in $L^2(\Omega)^d$} \\
& & \mathbf{V}^s:=\mbox{closure of $\mathcal{V}$ in $W^{s,2}(\Omega)^d$},
\end{eqnarray*}
where $s\geq 1$.
For $s=1$, we use the notation $\mathbf{V}$ instead of $\mathbf{V}^1$.
Similarly, we define the scalar function space
\begin{equation*}
V:=\mbox{closure of $C_0^{\infty}(\Omega)$ in $H^1(\Omega)$}.
\end{equation*}

\section{Main result}\label{Sect:main}

In the mathematical analysis of the turbulence problem (\ref{problem1})-(\ref{problem5}), there is a set of usual assumptions that, although they do not follow from the real situation, are physically admissible, %
\begin{equation}\label{e-visc-Carath}
\nut,\ \nu_D,\ \nu_P,\ \varepsilon,\ e:Q_T\times\mathds{R}\rightarrow\mathds{R}_0^+\quad\mbox{are Carathéodory functions}.
\end{equation}
The novelty of this work lies in the hypotheses that we state next.
On the functions of turbulent viscosity $\nut$, turbulent diffusion $\nu_D$, turbulence production $\nu_P$ and turbulent dissipation $\varepsilon$, we assume that, for certain constants $\eta,\ \zeta,\ \gamma,\ \vartheta \in \mathds{R}^+_0$, there exist couples of positive constants,
$c_T$, $C_T$, $c_D$, $C_D$, $c_P$, $C_P$ and $c_\varepsilon$, $C_\varepsilon$ such that
\begin{eqnarray}
 && c_T(1+k)^\eta\leq \nut(k)\leq C_T(1+k)^\eta, \label{f:visc-turb} \\
 && c_D(1+k)^\zeta\leq \nu_D(k)\leq C_D(1+k)^\zeta, \label{f:diff-turb} \\
 && c_P(1+k)^\gamma\leq \nu_P(k)\leq C_P(1+k)^\gamma, \label{f:P-turb} \\
 && c_\varepsilon k^{\vartheta+1}\leq \varepsilon(k)\leq C_\varepsilon k^{\vartheta+1} , \label{f:dissip-turb}
 \end{eqnarray}
for all $k \in \mathds{R}_0^+$ and a.e. in $Q_T$.

We assume on the external forces field that
\begin{equation}\label{g:V'}
\mathbf{g}\in L^2(0,T;L^2(\Omega)^d),
\end{equation}
and on the initial data that
\begin{eqnarray}
&& \label{eq:cond_ini_1}
{\bu}_0 \in \mathbf{H}, \\
&& \label{eq:cond_ini_2}
k_0 \in L^1(\Omega).
\end{eqnarray}
In addition, we assume the existence of a positive constant $C_0$ such that
\begin{equation}\label{k0>C}
k_0\geq C_0>0\quad \mbox{a.e. in}\ \Omega.
\end{equation}

To ensure that the terms containing the nonlinear functions $\nut(k)$, $\nu_D(k)$, $\nu_P(k)$ and $\varepsilon(k)$ are somewhat more than $L^1$--integrable (with the exception of the first r.h.s. term of \eqref{problem3} that is only in $L^1$), it is necessary to make some assumptions on the exponents of nonlinearity set in \eqref{f:visc-turb}-\eqref{f:dissip-turb}.
For this purpose, let us set
\begin{equation}\label{ru:rk}
r_u:=\max\left\{\frac{2(d+2)}{d},\alpha\right\}, \quad  \rho_k:=\max\left\{\frac{2(d+2)}{d},\vartheta+2\right\},\quad r_k:=\zeta+1+\frac{2}{d}.
\end{equation}
We assume that
\begin{eqnarray}
&& \label{eq:Cond1}  \eta < r_k
\end{eqnarray}
to make sure that $\nut(k)\mathbf{D}({\bu}) \in L^q (0,T;L^q(\Omega)^{d \times d})$ for some $q>1$.
To ensure that $\varepsilon(k) \in L^q(0,T;L^q(\Omega))$, for some $q>1$, we assume that
\begin{eqnarray}
\label{eq:Cond2}
&&  \vartheta<\zeta+\frac{2}{d}.
\end{eqnarray}
And to make sure that  $\nu_P (k) |{\bu} |^\beta  \in L^q(0,T;L^q(\Omega))$, for some $q>1$, we assume that
\begin{eqnarray}
&&
\label{Hyp:theta:gamma}
\frac{\gamma}{\vartheta+1 } + \frac{\beta}{r_u} < 1.
\end{eqnarray}

The main result of this work is written in the following theorem.

\begin{theorem}\label{thm:exist}
Let $\Omega$ be a bounded domain of $\mathbb{R}^d$, where it is supposed that $2\leq d\leq 4$ and $\partial\Omega$ is Lipschitz-continuous.
Assume (\ref{e-visc-Carath}), (\ref{f:visc-turb})-(\ref{f:dissip-turb}), (\ref{g:V'}), \eqref{eq:cond_ini_1}-\eqref{eq:cond_ini_2} and (\ref{k0>C}), and \eqref{eq:Cond1}-\eqref{Hyp:theta:gamma} hold true.
Then, there exists a couple of functions $({\bu},k)$ such that:
\begin{enumerate}[leftmargin=*,topsep=-5pt,label=(\arabic*)]
\item ${\bu}\in L^2(0,T;\mathbf{V})\cap L^{\infty}(0,T;\mathbf{H})\cap L^{r_u}(0,T;L^{r_u}(\Omega)^d)$ for $r_u$ given in (\ref{ru:rk});
\item $k\in L^{\infty}(0,T;L^1(\Omega))\cap L^q(0,T;W^{1,q}_0(\Omega)) \cap L^r(0,T;L^r(\Omega)) \cap L^{1+ \vartheta}(0,T;L^{1+ \vartheta}(\Omega))$, for
\begin{equation}\label{w:reg:k:W1q}
\left\{
\begin{array}{ll}
q=2, & \mbox{if}\ \ \zeta>1,\ \mbox{or} \\
1<q<1+\frac{d\zeta + 1}{d+1}, & \mbox{if}\ \ 0\leq\zeta\leq 1
\end{array}\right.
\qquad\mbox{and}\qquad 1<r<r_k,\quad \mbox{with $r_k$ given in \eqref{ru:rk}};
\end{equation}
\item $k\geq C_0$ a.e. in $Q_T$;
\item $\sqrt{\nut(k)}\mathbf{D}({\bu})\in L^2(0,T;L^2(\Omega)^{d\times d})$;
\item For every $\bm{\varphi}\in C^\infty(Q_T)^d$ such that $\operatorname{div}\bm{\varphi} =0$ in $Q_T$ and $\operatorname{supp}\bm{\varphi}\subset\subset \Omega\times[0,T)$,   there holds
\begin{equation}\label{weak-form-u}
\begin{split}
&
-\int_0^T \int_{\Omega} \bu\cdot\partial_t\bm{\varphi}\,dxdt -  \int_0^T \int_{\Omega}{\bu}(t)\otimes{\bu}(t):\nabla\bm{\varphi}\,dxdt
  +  \int_0^T \int_{\Omega}\nut(k)\,\mathbf{D}({\bu}):\nabla\bm{\varphi}\,dxdt  \\
&  +  \int_0^T\int_{\Omega}\left(c_{D}+c_F|{\bu}|^{\alpha-2}\right){\bu}\cdot\bm{\varphi}\,dxdt
=  \int_{\Omega} \bu_0\cdot\bm{\varphi}(0)\,dx + \int_0^T \int_\Omega \mathbf{g}\cdot\bm{\varphi}\,dxdt;
\end{split}
\end{equation}
\item For every $w\in C^\infty(Q_T)$ such that $w \geq 0$ a.e in $Q_T$ and $\operatorname{supp}w\subset\subset\Omega\times[0,T)$, there holds
\begin{equation}\label{weak-form-k}
\begin{split}
& -\int_0^T \int_{\Omega} k\partial_tw\,dxdt - \int_0^T\int_{\Omega}k{\bu}\cdot\nabla w\,dxdt
  + \int_0^T \int_{\Omega}\nu_D(k)\nabla k\cdot\nabla w\,dxdt + \int_0^T \int_{\Omega}\varepsilon(k)w\,dxdt \geq \\
&
\int_{\Omega} k_0w(0)\,dx +
\int_0^T \int_{\Omega}\nut(k)|\mathbf{D}({\bu})|^2w\,dxdt + \int_0^T \int_{\Omega}\nu_P(k)|{\bu}|^\beta w\,dxdt;
\end{split}
\end{equation}
\item The initial conditions are satisfied in the following sense
\begin{equation}\label{attain:ic:u0k0}
\lim_{t\to 0^+}\left(\left\|\bu(t)-\bu_0\right\|_2^2 + \left\|k(t)-k_0\right\|_1\right)=0.
\end{equation}
\end{enumerate}
Moreover,
\begin{enumerate}[leftmargin=*,topsep=-5pt,label=(\arabic*)]
\setcounter{enumi}{7}
  \item ${\bu}_t\in L^{\varsigma}(0,T;W^{-1,\varsigma}(\Omega)^d)$ for $1<\varsigma<\varsigma_0$, with $\varsigma_0$ defined below in \eqref{varsigma:00} (see also \eqref{varsigma:0});
  \item $k_t\in \mathcal{M}(0,T;W^{-1,\varrho}(\Omega))$
for $1<\varrho<\varrho_0$, with $\varrho_0$ defined below in \eqref{varrho:0} (see also \eqref{varrho:00}).
\end{enumerate}
\end{theorem}

In Theorem~\ref{thm:exist}, $\mathcal{M}(0,T;W^{-1,\varrho}(\Omega))$ denotes the space of Radon measures $\sigma:[0,T]\longrightarrow W^{-1,\varrho}(\Omega)$, where $W^{-1,\varrho}(\Omega)$ denotes the dual space of $W^{1,\varrho'}_0(\Omega)$, and $\varrho'$ is the Hölder conjugate of $\varrho$.

Observe that in \eqref{weak-form-u} the notion of solution is in the usual weak sense, but in \eqref{weak-form-k} the solution is considered in a suitable weak sense once the equality, for the best of author's knowledge, is not known how to be reached.
In a way, this resembles the notion of suitable weak solutions introduced in~\cite{CKN:1982}.
The notion of weak solution satisfying only \eqref{weak-form-u}-\eqref{weak-form-k}, without requiring an extra (opposite in)equality, may be considered very weak, because there can easily be many solutions in that conditions.
The alternative would be to proceed as in \cite{BLM:2011,BFM:2009}, considering Navier's slip boundary conditions, so that we can recover the pressure.
But then we would no longer be studying the same problem.
This issue has also been extensively studied in previous works of turbulence in clear flows~\cite{BLM:2011,D:2008,GLLMT:2003,LL:2007,L:1997,MN:2015}, but it has not yet been possible to solve it, in particular in the case of Dirichlet boundary conditions that we consider here.
The case of Navier's slip boundary conditions will be investigated shortly by the author.

Let us now make some comments regarding the enumerated items of the previous theorem, especially in the dimensions of physical interest $d=3$ and $d=2$.

\begin{remark}
1.
The range of $q$ assumed in (\ref{w:reg:k:W1q}) is required to prove estimates (\ref{est:grad:kj:L2:0}) and (\ref{est:grad:kj:Lq:0}) below (see also (\ref{w:conv:kn:W1q})).
In the case of $d=3$ or $d=2$, which correspond to the relevant situations from the point of view of physics, we get from (\ref{w:reg:k:W1q}), in the case of $0\leq\zeta\leq 1$, $q<\frac{3\zeta + 5}{4}$ if $d=3$, and $q<\frac{2\zeta+4}{3}$ if $d=2$.
The values of $q$ obtained here agree with~\cite{BLM:2011}, where, in addition to working only with turbulence in clear fluid flows, the authors of \cite{BLM:2011} solely considered the space dimension $d=3$.
Therefore the regularity
$k\in L^q(0,T;W^{1,q}_0(\Omega))$ is obtained whether we consider turbulence in fluid flows through permeable media or in clear fluid flows, being our work more general.
It should also be stressed that for $q$ given by either the cases in (\ref{w:reg:k:W1q}), we always have $q>d'$, which improves the range of $q$ considered in the works \cite{L:1997,O:2024a,OP:2017,OP:2018a}.

2. Observe that from \eqref{varsigma:0}, we can write
\begin{equation}\label{varsigma:00}
\varsigma_0:=
\left\{
\begin{array}{ll}
\min\left\{
\frac{2\left(d\zeta+d+2\right)}{d\zeta+d\eta+d+2},
1+\frac{2}{d},
\frac{1+\frac{2}{d}}{\alpha-1}\right\}, & \mbox{if}\ \ r_u=\frac{2(d+2)}{d}, \\
\min\left\{
\frac{2\left(d\zeta+d+2\right)}{d\zeta+d\eta+d+2},
\frac{\alpha}{2},
\frac{\alpha}{\alpha-1}\right\}, & \mbox{if}\ \ r_u=\alpha,
\end{array}
\right.
\end{equation}
Note that if $r_u=\alpha$, then $\alpha>2$. This, together with assumption \eqref{eq:Cond1}, assures us that $\varsigma_0>1$ in any case.
In the particular case of $r_u=\frac{2(d+2)}{d}$ and $\alpha<2$ in (\ref{varsigma:00}), then it would come
\begin{equation*}
\varsigma_0=\min\left\{\frac{2(d\zeta +d+2)}{d\zeta+d\eta+d+2},1+\frac{2}{d}\right\}
=
\left\{
\begin{array}{ll}
\min\left\{\frac{2(3\zeta+5)}{3\zeta+3\eta+5},\frac{5}{3}\right\}, & \mbox{if}\ \ d=3, \\
\min\left\{\frac{2\zeta+4}{\zeta+\eta+2},2\right\}, & \mbox{if}\ \ d=2,
\end{array}
\right.
\end{equation*}
which, in the case of $d=3$, was precisely the value obtained in \cite{BLM:2011} to show that both $\partial_t{\bu}$ and $p$ are in
$L^{\varsigma}(0,T;W^{-1,\varsigma}(\Omega)^d)$ for $1<\varsigma<\varsigma_0$.

3. From \eqref{varrho:0}, we can write
\begin{equation}\label{varrho:00}
\varrho_0=\left\{
\begin{array}{ll}
\min\left\{\frac{d\zeta+d+2}{d\zeta+d+1},\frac{2(d\zeta+d+2)(d+2)}{d(d\zeta+3d+6)},\frac{2(d\zeta+d+2)(d+2)}{d\beta(d\zeta+d+2)+ 2\gamma d(d+2)}\right\}, & \mbox{if}\ \ r_u=\frac{2(d+2)}{d} \\
\min\left\{\frac{d\zeta+d+2}{d\zeta+d+1},\frac{\alpha(d\zeta+d+2)}{d\zeta+d\alpha+d+2},\frac{(d\zeta +d+2)\alpha}{\beta(d\zeta +d+2)+ \gamma d \alpha}\right\}, & \mbox{if}\ \ r_u=\alpha.
\end{array}
\right.
\end{equation}
From assumption \eqref{Hyp:theta:gamma}, $\rho_0>1$ in any case.
If $r_u=\frac{2(d+2)}{d}$, then
\begin{equation*}
\varrho_0=
\left\{
\begin{array}{ll}
\min\left\{\frac{3\zeta+5}{3\zeta+4},\frac{10}{9}\frac{3\zeta+5}{\zeta+5},\frac{10(3\zeta+5)}{3\beta(3\zeta+5)+30\gamma}\right\}\leq
\min\left\{\frac{3\zeta+5}{3\zeta+4},\frac{10}{9}\frac{3\zeta+5}{\zeta+5}\right\}, & \mbox{if}\ d=3, \\
\min\left\{\frac{2\zeta+4}{2\zeta+3},\frac{4(\zeta+2)}{\zeta+6},\frac{4(\zeta+2)}{\beta(\zeta+2)+4\gamma}\right\}\leq
\min\left\{\frac{2\zeta+4}{2\zeta+3},\frac{4(\zeta+2)}{\zeta+6}\right\}, & \mbox{if}\ d=2,
\end{array}
\right.
\end{equation*}
which, again in the case of $d=3$, was the precise value obtained in \cite{BLM:2011} to justify the boundedness of
$k_t$ in $\mathcal{M}(0,T;W^{-1,\varrho}(\Omega))$ for  $1<\varrho<\varrho_0$.
\end{remark}

For the sake of organization, the proof of Theorem~\ref{thm:exist} shall be split into the sections that follow.
We shall first consider an auxiliary problem that not only truncates all the nonlinear turbulence terms but also regularizes the convective term.

\section{Truncated problem}\label{Sect:trunc}

As the term $\nut(k)|\mathbf{D}({\bu})|^2$ is only in $L^1$ and since, in this work, we are considering the coefficient functions $\nut(k)$, $\nu_D(k)$ and $\nu_P(k)$ with increasingly larger values, we start by considering an approximate problem that takes into account the truncation of these terms.
Let
 $\mathcal{T}_n:\mathds{R}\longrightarrow\mathds{R}$ denote the truncation function at height $n$, given by
\begin{equation}\label{trunc:T}
\mathcal{T}_n (k) =
\left \{ \begin{array}{ll} k & \hbox{if } |k| \leq n, \\
\frac{n}{| k | } k & \hbox{if } |k| > n,
\end{array} \right.
\end{equation}
and let
\begin{equation} \nut^{(n)} =  \nut\circ \mathcal{T}_n, \quad \nu_D^{(n)} =  \nu_D \circ \mathcal{T}_n, \quad  \nu_P^{(n)} =  \nu_P \circ \mathcal{T}_n.
\label{comp:n}
\end{equation}
Note that, in view of (\ref{f:visc-turb}), (\ref{f:diff-turb}) and (\ref{f:P-turb}), one has
\begin{alignat}{2}
 c_T\leq\nut^{(n)}(k)\leq &\ C_T(1+n)^\eta,  && \label{nuT:n:bd} \\
 c_D\leq\nu_D^{(n)}(k)\leq &\ C_D(1+n)^\zeta,  \label{nuD:n:bd} && \\
 c_P\leq\nu_P^{(n)}(k)\leq &\ C_P(1+n)^\gamma  \label{nuP:n:bd} &&
\end{alignat}
for all $k \in \mathds{R}^+_0$ and a.e. in $Q_T$.

Let us now extend $k_0$ to the whole $\mathds{R}^d$ in such a way that, for this extension, say $\overline{k}_0$, $\overline{k}_0=C_0$ in $\mathds{R}^d\setminus\Omega$, and where $C_0$ is the positive constant from assumption \eqref{k0>C}.
Next, we regularize $\overline{k}_0$ by considering its mollifying function
\begin{equation}\label{moll:k0}
k_{n,0}:=\eta_\delta\star \overline{k}_0,,\qquad \delta=n^{-1},\quad n\in\mathds{N},
\end{equation}
where  $\eta_\delta$ is the Friedrichs mollifying kernel.
In view of assumption \eqref{k0>C}, one has
$k_{n,0} \geq C_0>0$ 
a.e. in $\Omega$.
In addition, due to \eqref{eq:cond_ini_2} and \eqref{moll:k0},
\begin{equation}\label{sc:rn00}
k_{n,0}\xrightarrow[n\to\infty]{}k_{0} \ \ \mbox{in}\ \ L^1(\Omega).
\end{equation}

We consider a sequence $\bu_{n,0}\in\bH$ such that
\begin{equation}\label{conv:un:u0}
\bu_{n,0}\xrightarrow[n\to\infty]{}\bu_{0} \ \ \mbox{in}\ \ L^2(\Omega)^d.
\end{equation}

To be able to use the energy equality of the mean flow equation in the final stage of the proof of the Theorem~\ref{thm:exist} (see \eqref{weak-form-u:n:sqrt} later), we regularize the velocity field in the convective term.
For that, let $\Phi\in C^\infty\big([0,\infty)\big)$ be a non-increasing function such that
\begin{equation}\label{Phi}
\Phi(\tau)=\left\{
\begin{array}{ll}
1 & \mbox{if}\ \ 0\leq\tau\leq 1, \\
0 & \mbox{if}\ \ \tau\geq 2,
\end{array}\right.
\qquad 0\leq \Phi\leq 1\ \ \mbox{in}\ \ [0,\infty)
\end{equation}
For $n\in\mathds{N}$, we set
\begin{equation}\label{Phi:d}
\Phi_n(\tau)=\Phi\left(\frac{\tau}{n}\right),\qquad \tau\in[0,\infty).
\end{equation}

For each $n\in\mathds{N}$, we consider the truncated and regularized problem
\begin{alignat}{2}
& \partial_t{\bu} + \operatorname{div}\big(\Phi_n(|\bu|^2)\bu(t)\otimes\bu(t)\big) - \mathbf{div}\left(\nut^{(n)}(k)\mathbf{D}({\bu})\right) + \nabla p=\mathbf{g}- \Big(c_{D}+c_F|{\bu}|^{\alpha-2}\Big) {\bu} \quad \mbox{in}\ \ Q_T , \label{problem1:n}\\
& \operatorname{div}{\bu}=0 \qquad \mbox{in}\ \ Q_T,\label{problem2:n}\\
& \partial_tk+{\bu}\cdot\nabla k -\operatorname{div}(\nu_D^{(n)}(k)\nabla k)=\nut^{(n)}(k)|\mathbf{D}({\bu})|^2 +\nu_P^{(n)} (k)|{\bu}|^\beta-\varepsilon(k)\quad \mbox{in}\ \ Q_T,\label{problem3:n}\\
& {\bu} ={\bu}_{n,0}\quad\mbox{and}\quad k =k_{n,0} \quad \mbox{in}\ \ \Omega\times\{0\}, \label{problem4:n}\\
& {\bu} =\mathbf{0}\quad\mbox{and}\quad k=0 \quad \mbox{on}\ \ \Gamma_T.  \label{problem5:n}
\end{alignat}

The next result asserts the existence of truncated-regularized solutions  to the problem (\ref{problem1})-(\ref{problem5}).

\begin{proposition}\label{prop:exist:n}
Let the conditions of Theorem~\ref{thm:exist} be fulfilled.
Then (for each $n\in\mathds{N}$) there exists, at least, a couple of solutions $({\bu}_n,k_n)$ to the problem (\ref{problem1:n})-(\ref{problem5:n})
such that (1)-(3) and (7) of Theorem~\ref{thm:exist} are fulfilled and for every  $\bv\in \mathbf{V}\cap L^\alpha(\Omega)^d$ and every $w\in W^{1,\infty}_0(\Omega)$,
\begin{equation}
\begin{split}\label{weak-form-u:n}
& \frac{d}{dt}\int_{\Omega} \bu_n(t)\cdot\bv\,dx - \int_{\Omega}\Phi_n(|\bu_n|^2)\bu_n(t)\otimes\bu_n(t):\nabla\bv\,dx
  +\int_{\Omega}\nut^{(n)}(k_n(t))\,\mathbf{D}(\bu_n(t)):\nabla\bv\,dx \\
&  +\int_{\Omega}\left(c_{Da}+c_{Fo}|\bu_n(t)|^{\alpha-2}\right)\bu_n(t)\cdot\bv\,dx
=
\int_{\Omega}\bg(t)\cdot\bv\,dx
\end{split}
\end{equation}
and
\begin{equation}\label{weak-form-k:n}
\begin{split}
& \frac{d}{dt}\int_{\Omega} k_n(t)w\,dx - \int_{\Omega}k_n(t)\bu_n(t)\cdot\nabla w\,dx
  +\int_{\Omega}\nu_D^{(n)}(k_n(t))\nabla k_n(t)\cdot\nabla w\,dx  + \int_{\Omega}\varepsilon(k_n(t))w\,dx  \\
& = \int_{\Omega}\nut^{(n)}(k_n(t))|\mathbf{D}(\bu_n(t))|^2w\,dx + \int_{\Omega}\nu_P^{(n)}(k_n(t))|\bu_n(t)|^\beta w\,dx
\end{split}
\end{equation}
hold for all $t\in(0,T)$.
\end{proposition}

\begin{proof}(Proposition~\ref{prop:exist:n})
The proof of Proposition~\ref{prop:exist:n} will be carried out in the next sections.

\section{Galerkin approximations for the truncated problem}\label{Sect-GA-tr}

In this section, we start the proof of Proposition~\ref{prop:exist:n}.
For the sake of simplifying de notation, in the course of this proof, we drop the subscript $n$.

We proceed as in~\cite{O:2024a,O:2023a} and consider orthogonal bases $\left\{\bv_{i}\right\}_{i\in\mathds{N}}$ of $\mathbf{V}^s$, for $s>1+\frac{d}{2}$, and $\left\{w_{i}\right\}_{i\in\mathds{N}}$ of $H^1(\Omega)$, for $s>\frac{d}{2}$,  that are orthonormal in $L^2(\Omega)^d$ and in $L^2(\Omega)$, respectively.
Given $j,\ l\in\mathds{N}$, let us consider the
$j$-dimensional space $\mathbf{X}_{j}:=\operatorname{span}\{\bv_{1},\dots,\bv_{j}\}$ and the $l$-dimensional space $X_{l}:=\operatorname{span}\{w_{1},\dots,w_{l}\}$.
For each $j\in\mathds{N}$ and  each  $l\in\mathds{N}$, we search for approximate solutions
\begin{alignat}{2}
& {\bu}^{j,l}(x,t)= \sum_{i=1}^{j}a_i^{j,l}(t)\bv_{i}(x),\quad\bv_{i}\in\mathbf{X}_{j}, && \label{eq:pres1} \\
& k^{j,l}(x,t)= \sum_{i=1}^{l}c_i^{j,l}(t)w_{i}(x),\quad w_{i}\in X_{l},\qquad  k^{j,l}\geq 0\quad\mbox{a.e. in}\ Q_T, && \label{eq:pres2}
\end{alignat}
where the coefficients $a_{1}^{j,l}(t),\dots,a_{j}^{j,l}(t)$ and $c_{1}^{j,l}(t),\dots,c_{l}^{j,l}(t)$ solve the following system of $j+l$ ordinary differential equations
\begin{equation}\label{weak-form-u:tr:ap}
\begin{split}
&
\frac{d}{dt}\int_{\Omega} {\bu}^{j,l}(t)\cdot\bv_i\,dx - \int_{\Omega}\Phi_n(|\bu_n|^2){\bu}^{j,l}(t)\otimes{\bu}^{j,l}(t):\nabla\bv_i\,dx
  +\int_{\Omega}\nut^{(n)}(k^{j,l}(t))\,\mathbf{D}({\bu}^{j,l}(t)):\mathbf{D}(\bv_i)\,dx \\
&  +\int_{\Omega}\left(c_{Da}+c_{Fo}|{\bu}^{j,l}(t)|^{\alpha-2}\right){\bu}^{j,l}(t)\cdot\bv_i\,dx
 =
\int_\Omega \mathbf{g}(t)\cdot\bv_i\,dx, \qquad i=1,\dots,j,
\end{split}
\end{equation}
\begin{equation}\label{weak-form-k:tr:ap}
\begin{split}
& \frac{d}{dt}\int_{\Omega} k^{j,l}(t)w_i\,dx - \int_{\Omega}k^{j,l}(t){\bu}^{j,l}(t)\cdot\nabla w_i\,dx
  +\int_{\Omega}\nu_D^{(n)}(k^{j,l}(t))\nabla k^{j,l}(t)\cdot\nabla w_i\,dx \\
& + \int_{\Omega}\varepsilon(k^{j,l}(t))w_i\,dx = \\
& \int_{\Omega}\nut^{(n)}(k^{j,l}(t))|\mathbf{D}({\bu}^{j,l})|^2w_i\,dx + \int_{\Omega}\nu_P^{(n)}(k^{j,l}(t))|{\bu}^{j,l}(t)|^\beta w_i\,dx,
\qquad i=1,\dots,l.
\end{split}
\end{equation}
System (\ref{weak-form-u:tr:ap})-(\ref{weak-form-k:tr:ap}) is
supplemented with the following initial conditions
\begin{equation}\label{eq:0}
\bu^{j,l}(0)=\bu^{j,l}_{0}\quad
\mbox{and}\quad k^{j,l}(0)=k^{j,l}_{0}\quad\mbox{in}\ \Omega,
\end{equation}
where $\bu^{j,l}_{0}$ and $k^{j,l}_{0}$ are the orthogonal projections of $\bu_{n,0}$ and $k_{n,0}$ onto $\mathbf{X}_{j}$ and $X^l$, respectively.
Whence
\begin{equation*}
\bu^{j,l}_{0}= \sum_{i=1}^{j}a_{0,i}^{j,l}\bv_{i},\quad\bv_{i}\in\mathbf{X}_{j},\qquad 
k^{j,l}_0= \sum_{i=1}^{l}c_{0,i}^{j,l}w_{i},\quad w_{i}\in X_{l},  
\end{equation*}
for some $\ba_0=(a_{0,1}^{j,l},\dots,a_{0,j}^{j,l})\in\R^j$ and $\bc_0=(c_{0,1}^{j,l},\dots,c_{0,l}^{j,l})\in\R^l$.
We can assume that
\begin{alignat}{2}
& \label{sc:vn0}
\bu^{j,l}_{0}\xrightarrow[l\to\infty]{} \bu^{j}_{0} \ \ \mbox{in}\ \ L^2(\Omega)^d,\qquad
&& \bu^{j}_{0}\xrightarrow[j\to\infty]{} \bu_{n,0}\ \ \mbox{in}\ \ L^2(\Omega)^d,
\\
&
\label{sc:rn0}
k^{j,l}_{0}\xrightarrow[l\to\infty]{} k_0^j \ \ \mbox{in}\ \ L^2(\Omega),\qquad
&& k^{j}_{0}\xrightarrow[j\to\infty]{}k_{n,0} \ \ \mbox{in}\ \ L^1(\Omega).
\end{alignat}
The existence of solutions $\ba(t)=(a_{1}^{j,l}(t),\dots,a_{j}^{j,l}(t))$ and $\bc(t)=(c_{1}^{j,l}(t),\dots,c_{l}^{j,l}(t))$ solving the Cauchy problem (\ref{weak-form-u:tr:ap})-(\ref{eq:0}) in the entire time interval $[0,T]$ is justified by the application of Carathéodory's theorem and the Continuation Principle (see~\cite{O:2024a} for the details).

Using assumptions (\ref{f:P-turb}), (\ref{g:V'}), (\ref{Hyp:theta:gamma}) and (\ref{eq:cond_ini_1})-(\ref{eq:cond_ini_2}), together with the boundedness of the truncated turbulent-depending functions, set in (\ref{nuT:n:bd})-(\ref{nuP:n:bd}), we can proceed as in~\cite{O:2024a}, to prove that
\begin{alignat}{2}
& \label{est2:eq:u:T}
\sup_{t\in[0,T]}\|{\bu}^{j,l}(t)\|_{2}^2 + \int_0^T\left\|\sqrt{\nut^{(n)}(k^{j,l}(t))}\mathbf{D}(\bu^{j,l}(t))\right\|_{2}^2dt + c_{Fo}\int_0^T\|{\bu}^{j,l}(t)\|_{\alpha}^{\alpha}dt
\leq C_1, && \\
& \label{est3:eq-u:jl}
\int_0^T\|\nabla{\bu}^{j,l}(t)\|_{2}^2dt \leq C_2, && \\
& \label{est4:eq-u:jl}
\int_0^T\|{\bu}^{j,l}(t)\|_{r_u}^{r_u}dt\leq C_3,\qquad \mbox{$r_u$ given in (\ref{ru:rk})}, && \\
& \label{est:dudt:jl:2}
\int_0^T\|\pt\bu^{j,l}(t)\|_{2}^2dt \leq C_4(n,j), && \\
& \label{est:dadt:jl}
\int_0^T\left|\frac{d\ba^{j,l}(t)}{dt}\right|_{2}^2dt \leq C_5(n,j) &&
\end{alignat}
and
\begin{alignat}{2}
& \label{est4_0:eq-k}
\sup_{t\in[0,T]}\|k^{j,l}(t)\|_{2}^2  + c_\varepsilon\int_0^T\|k^{j,l}(t)\|_{\vartheta+2}^{\vartheta+2}dt +  \int_0^T\left\|\sqrt{\nu_D^{(n)}(k^{j,l}(t))}\nabla k^{j,l}(t)\right\|^2_2dt \leq C_6(n,j), && \\
& \label{est5:eq-k}
\int_0^T\|\nabla k^{j,l}(t)\|^2_2dt \leq C_7(n,j), && \\
& \label{est4:eq-k:t}
\int_0^T\|k^{j,l}(t)\|_{\rho_k}^{\rho_k}dt  \leq C_8(n,j),\qquad \mbox{$\rho_k$ given in (\ref{ru:rk})}, && \\
& \label{est:dt:k:jl}
\int_0^T\left\|\partial_tk^{j,l}(t)\right\|_{W^{-1,s}(\Omega)}^{\rho}dt\leq C_9(n,j),\qquad \rho:=\min\left\{\frac{r_u}{\beta},\frac{\rho_k}{\vartheta+1}\right\},\quad s>\frac{d}{2}, &&
\end{alignat}
for some positive constants $C_1,\dots,C_9$.

Then, due to the uniform (independent of $j$), estimates (\ref{est2:eq:u:T})-(\ref{est4:eq-u:jl}) and (\ref{est4_0:eq-k})-(\ref{est:dt:k:jl}), we can combine the Banach-Alaoglu theorem with the Aubin-Lions compactness lemma and the Riesz-Fischer theorem to extract subsequences (still labeled by the same superscript $l$) such that
\begin{alignat}{3}
& \label{w*:conv:ujl}
{\bu}^{j,l}\xrightharpoonup[l\to\infty]{\ast} {\bu}^j\ \ \mbox{in}\ \ L^\infty(0,T;\mathbf{H}), && \\
  & \label{w:conv:ujl}
{\bu}^{j,l}\xrightharpoonup[l\to\infty]{} {\bu}^j,\ \ \mbox{in}\ \ L^2(0,T;\mathbf{V})\cap L^{r_u}(0,T;L^{r_u}(\Omega)^d), && \\
  & \label{w:conv:ajl'}
\ba^{j,l}\xrightharpoonup[l\to\infty]{} \ba^j\ \ \mbox{in}\ \ W^{1,2}(0,T), &&
\end{alignat}

\begin{alignat}{2}
& \label{w*:conv:kjl}
k^{j,l}\xrightharpoonup[l\to\infty]{\ast} k^j\ \ \mbox{in}\ \ L^\infty(0,T;L^2(\Omega)), && \\
  & \label{w:conv:kjl}
k^{j,l}\xrightharpoonup[l\to\infty]{} k^j\ \ \mbox{in}\ \ L^2(0,T;H^1_0(\Omega))\cap L^{\rho_k}(0,T;L^{\rho_k}(\Omega)), && \\
& \label{w:conv:kjl'}
\partial_tk^{j,l}\xrightharpoonup[l\to\infty]{} \partial_tk^j\ \ \mbox{in}\ \ L^2(0,T;W^{-s,2}(\Omega)),\qquad
s> \frac{d}{2},
\end{alignat}
\begin{alignat}{2}
& \label{str:conv:ujl}
{\bu}^{j,l}\xrightarrow[l\to\infty]{}{\bu}^j\ \ \mbox{in}\ \ L^q(0,T;L^q(\Omega)^d)\quad \forall\ q:\ 1\leq q<r_u,  && \\
  & \label{str:conv:kjl}
k^{j,l}\xrightarrow[l\to\infty]{}k^j\ \ \mbox{in}\ \ L^q(0,T;L^q(\Omega))\quad \forall\ q:\ 1\leq q<\rho_k, && \\
& \label{str:conv:ajl'}
\ba^{j,l}\xrightarrow[l\to\infty]{}\ba^j\ \ \mbox{in}\ \ C[0,T] &&
\end{alignat}
and
\begin{alignat}{2}
& \label{ae:conv:ujl}
{\bu}^{j,l}\xrightarrow[l\to\infty]{}{\bu}^j\quad \mbox{a.e.\ \ in}\ \ Q_T,  && \\
& \label{ae:conv:kjl}
k^{j,l}\xrightarrow[l\to\infty]{}k^j\quad \mbox{a.e.\ \ in}\ \ Q_T, && \\
& \label{unif:conv:gradu:jl}
\nabla\bu^{j,l}\xrightarrow[l\to\infty]{\ \text{uniformly}\ } \nabla\bu^j\ \ \mbox{in}\ \ Q_T, &&
\end{alignat}
where $r_u$ and $\rho_k$ are given in (\ref{ru:rk}).

All the terms in the approximate mean flow equation (\ref{weak-form-u:tr:ap}), with the exception of the ones involving the Darcy and Forchheimer terms and the turbulent viscosity, can be proven to converge, as in the classical Navier-Stokes equations.
With respect to the drag forces, just the Forchheimer term needs some justification, which can be done exactly the same way as in~\cite{O:2024a}.

For the turbulent viscosity term, we first observe that, by (\ref{e-visc-Carath}), (\ref{ae:conv:kjl}) and \eqref{unif:conv:gradu:jl}, we have
\begin{equation}\label{ae:conv:nuT(k)Du:jl:1}
\nut(k^{j,l})\Dujl\xrightarrow[l\to\infty]{} \nut(k^{j})\Duj\quad \mbox{a.e.\ \ in}\ \ Q_T.
\end{equation}
Moreover, from \eqref{nuT:n:bd} and \eqref{est2:eq:u:T}, one has
\begin{equation*}
\int_0^T\|\nut(k^{j,l})\Dujlt\|_{2}^2dt \leq C(n),
\end{equation*}
for some positive constant $C$.
As a consequence,
\begin{equation}\label{w:conv:nuT(k)Dv:jl}
\nut(k^{j,l})\Dujl\xrightharpoonup[l\to\infty]{}  \nut(k^j)\mathbf{D}(\bu^j)\ \ \mbox{in}\ \ L^2(0,T;L^{2}(\Omega)^{d\times d}).
\end{equation}

With respect to the regularized convective term, we can use \eqref{Phi:d} and \eqref{ae:conv:ujl}, and proceed as in the proof of~\cite[Proposition~1]{O:2024a}, to show that
\begin{equation}\label{w:conv:Phi-d:jl}
\Phi_n(|\bu^{j,l}|^2)\bu^{j,l}\otimes\bu^{j,l}\xrightharpoonup[l\to\infty]{}
\Phi_n(|\bu^{j}|^2)\bu^{j}\otimes\bu^{j}\ \ \mbox{in}\ \ L^{\frac{r_u}{2}}(0,T;L^{\frac{r_u}{2}}(\Omega)^{d\times d}).
\end{equation}

Using the convergence results \eqref{w:conv:ujl}-\eqref{w:conv:ajl'} and \eqref{w:conv:nuT(k)Dv:jl}-\eqref{w:conv:Phi-d:jl}, we can pass to the limit $l\to\infty$ in the approximate weak formulation (\ref{weak-form-u:tr:ap}) to obtain
\begin{equation}
\begin{split}\label{weak-form-u:n:j}
& \frac{d}{dt}\int_{\Omega} {\bu}^j(t)\cdot\bv_i\,dx - \int_{\Omega}\Phi_n(|\bu^{j}(t)|^2){\bu}^j(t)\otimes{\bu}^j(t):\nabla\bv_i\,dx
  +\int_{\Omega}\nut^{(n)}(k^j(t))\,\mathbf{D}({\bu}^j(t)):\mathbf{D}(\bv_i)\,dx \\
&  +\int_{\Omega}\left(c_{Da}+c_{Fo}|{\bu}^j(t)|^{\alpha-2}\right){\bu}^j(t)\cdot\bv_i\,dx
=
\int_\Omega \mathbf{g}(t)\cdot\bv_i\,dx\qquad \forall\ i\in\{1,\dots,j\}.
\end{split}
\end{equation}

Regarding the approximate TKE equation (\ref{weak-form-k:tr:ap}), we just comment on the turbulent terms of diffusion, dissipation, viscosity and production.
Arguing as we did for (\ref{w:conv:nuT(k)Dv:jl}), but now using (\ref{nuD:n:bd}), (\ref{est4_0:eq-k}) and (\ref{ae:conv:kjl}), we can prove that
\begin{equation}\label{w:conv:nuD(k)Dk:jl}
\nu_D(k^{j,l})\nabla k^{j,l}\xrightharpoonup[l\to\infty]{}  \nu_D(k^j)\nabla k^j\ \ \mbox{in}\ \ L^2(0,T;L^{2}(\Omega)^{d}).
\end{equation}

Due to (\ref{e-visc-Carath}) and (\ref{ae:conv:kjl}), there holds
\begin{equation}\label{conv:ae:e(kjl)}
\varepsilon(k^{j,l})\xrightarrow[l\to\infty]{}  \varepsilon(k^{j})\ \ \mbox{a.e. in}\ \ Q_T,
\end{equation}
and by using assumption (\ref{f:dissip-turb}), together with (\ref{est4_0:eq-k}), we can show that
\begin{equation}\label{est4:eq-k:t+2}
\int_0^T\|\varepsilon(k^{j,l}(t))\|_{\frac{\vartheta+2}{\vartheta+1}}^{\frac{\vartheta+2}{\vartheta+1}}dt  \leq C(j,n)
\end{equation}
for some positive constant $C$.
Hence, (\ref{conv:ae:e(kjl)}) and (\ref{est4:eq-k:t+2}) assure that
\begin{equation}\label{w:conv:e(kjl)}
\varepsilon(k^{j,l})\xrightharpoonup[l\to\infty]{}  \varepsilon(k^{j})\ \ \mbox{in}\ \ L^{\frac{\vartheta+2}{\vartheta+1}}(0,T;L^{\frac{\vartheta+2}{\vartheta+1}}(\Omega)).
\end{equation}
Arguing as we did for (\ref{ae:conv:nuT(k)Du:jl:1}), we also have
\begin{equation*}
\nut^{(n)}(k^{j,l})|\Dujl|^2\xrightarrow[l\to\infty]{} \nut^{(n)}(k^{j})\big|\Duj\big|^2\quad \mbox{a.e.\ \ in}\ \ Q_T.
\end{equation*}
From \eqref{est2:eq:u:T}, we can show that
\begin{equation*}
\limsup_{l\to\infty}\int_0^T\int_\Omega\nut^{(n)}(k^{j,l})|\Dujl|^2\,dxdt \leq C,
\end{equation*}
for some positive constant $C$.
Therefore, in view of the Vitali-Hahn-Saks theorem,
\begin{equation}\label{s:conv:nuT(k)Dv:jl:L2}
\nut^{(n)}(k^j)|\Dujl|^2\xrightarrow[l\to\infty]{} \nut^{(n)}(k^j)|\mathbf{D}(\bu^j)|^2\ \ \mbox{in}\ \ L^1(0,T;L^1(\Omega)).
\end{equation}
On the other hand, from (\ref{e-visc-Carath}) and (\ref{ae:conv:ujl})-(\ref{ae:conv:kjl}), one has
\begin{equation}\label{ae:conv:nuP(kjl)ub:jl}
\nu_P^{(n)}(k^{j,l})|\bu^{j,l}|^\beta\xrightarrow[l\to\infty]{}\nu_P^{(n)}(k^{j})|\bu^{j}|^\beta\quad \mbox{a.e.\ \ in}\ \ Q_T.
\end{equation}
Using (\ref{nuP:n:bd}) and (\ref{est4:eq-u:jl}), we can show that
\begin{equation}\label{bd:nuP(kjl)ujb}
\int_0^T\int_\Omega \left|\nu_P^{(n)}(k^{j,l})|\bu^{j,l}|^\beta\right|^qdxdt\leq C(n)\quad \forall\ q: 1< q\leq\frac{r_u}{\beta},
\end{equation}
for some positive constant $C$.
Note that, due to assumption \eqref{Hyp:theta:gamma}, $\frac{r_u}{\beta}>1$.
As a consequence of (\ref{ae:conv:nuP(kjl)ub:jl}) and (\ref{bd:nuP(kjl)ujb}) there holds
\begin{equation}\label{w:conv:nuP(kjl)ujlb}
\nu_P^{(n)}(k^{j,l})|\bu^{j,l}|^\beta\xrightharpoonup[l\to\infty]{}  \nu_P^{(n)}(k^{j})|\bu^{j}|^\beta\ \ \mbox{in}\ \ L^q(0,T;L^q(\Omega))\quad \forall\ q: 1< q\leq\frac{r_u}{\beta}.
\end{equation}
Finally, we use \eqref{w:conv:kjl}-(\ref{w:conv:kjl'}), (\ref{w:conv:nuD(k)Dk:jl}), (\ref{w:conv:e(kjl)}), (\ref{s:conv:nuT(k)Dv:jl:L2}) and (\ref{w:conv:nuP(kjl)ujlb}), to pass to the limit $l\to\infty$ in the approximate weak formulation (\ref{weak-form-k:tr:ap}) so that
\begin{equation}\label{weak-form-k:tr:j}
\begin{split}
& \frac{d}{dt}\int_{\Omega} k^{j}(t)w_i\,dx - \int_{\Omega}k^{j}(t){\bu}^{j}(t)\cdot\nabla w_i\,dx
  +\int_{\Omega}\nu_D^{(n)}(k^{j}(t))\nabla k^{j}(t)\cdot\nabla w_i\,dx \\
& + \int_{\Omega}\varepsilon(k^{j}(t))w_i\,dx = \\
& \int_{\Omega}\nut^{(n)}(k^j)|\mathbf{D}({\bu}^j)|^2w_i\,dx + \int_{\Omega}\nu_P^{(n)}(k^{j}(t))|{\bu}^{j}(t)|^\beta w_i\,dx
\qquad \forall\ i\in\mathds{N}.
\end{split}
\end{equation}

By a classical reasoning (see e.g.~\cite[Ch.~III \S 3.2]{Temam:1979}), we can use (\ref{weak-form-u:tr:ap}), (\ref{weak-form-u:n:j}) with (\ref{eq:0})$_1$, (\ref{sc:vn0}), (\ref{str:conv:ujl}), (\ref{w:conv:nuT(k)Dv:jl}),  from one hand, and (\ref{weak-form-k:tr:ap}), (\ref{weak-form-k:tr:j}) with (\ref{eq:0})$_2$, (\ref{sc:rn0}), (\ref{str:conv:kjl}), (\ref{w:conv:nuD(k)Dk:jl}), (\ref{w:conv:e(kjl)}), (\ref{s:conv:nuT(k)Dv:jl:L2}), (\ref{w:conv:nuP(kjl)ujlb}), on the other, to show that
\begin{equation}\label{eq:0:uj:kj}
{\bu}^{j}(0)={\bu}^{j}_{0}\quad\mbox{and}\quad  k^{j}(0)=k^{j}_{0}\quad\mbox{in}\ \Omega.
\end{equation}

Now we can proceed as in~\cite{O:2024a} to show that (\ref{f:diss-turb2}) implies
\begin{equation} \label{f:diss-turb2:j}
e(k^{j})\geq 0\quad \mbox{a.e. in}\ Q_T,
\end{equation}
and that \eqref{f:diss-turb2:j}, along with (\ref{epslion(k):e(k)}), (\ref{weak-form-u:n:j}) and  (\ref{weak-form-k:tr:j}), allow us to prove that
\begin{equation}\label{kj>=0}
k^{j}\geq 0\quad\mbox{a.e. in}\ Q_T.
\end{equation}
And as a consequence of (\ref{epslion(k):e(k)}) and (\ref{kj>=0}), one has
\begin{equation}\label{epsilon(kj)>=0}
\varepsilon(k^{j})\geq 0\quad\mbox{a.e. in}\ Q_T.
\end{equation}

In the next section we aim to obtain estimates that are independent of $j$.

\section{Estimates independent of $j$}

Let us first obtain estimates for ${\bu}^{j}$, $\nabla{\bu}^{j}$,  and $\partial_t{\bu}^{j}$ that are independent of $j$.
By linearity and continuity, we can show from (\ref{weak-form-u:n:j}) that
\begin{equation}
\begin{split}\label{weak-form-u:n:j:v}
& \frac{d}{dt}\int_{\Omega} {\bu}^j(t)\cdot\bv\,dx - \int_{\Omega}\Phi_n(|\bu^{j}(t)|^2){\bu}^j(t)\otimes{\bu}^j(t):\nabla\bv\,dx
  +\int_{\Omega}\nut^{(n)}(k^j(t))\,\mathbf{D}({\bu}^j(t)):\mathbf{D}(\bv)\,dx \\
&  +\int_{\Omega}\left(c_{Da}+c_{Fo}|{\bu}^j(t)|^{\alpha-2}\right){\bu}^j(t)\cdot\bv\,dx
=
\int_\Omega \mathbf{g}(t)\cdot\bv\,dx
\end{split}
\end{equation}
holds for all $t\in(0,T)$ and all $\bv\in \mathbf{V}\cap L^\alpha(\Omega)^d$.
At any time $t\in(0,T]$, we take $\bv={\bu}^{j}(t)$ in (\ref{weak-form-u:n:j:v}) so that
\begin{equation}\label{est1:eq-u:j}
\begin{split}
& \frac{1}{2}\frac{d}{dt}\|{\bu}^{j}(t)\|_{2}^2
+\int_{\Omega}\nut^{(n)}(k^{j}(t))|\mathbf{D}({\bu}^{j}(t))|^2\,dx
+\int_\Omega\left(c_{Da}+c_{Fo}|{\bu}^{j}(t)|^{\alpha-2}\right)|{\bu}^{j}(t)|^2\,dx \\
&
=\int_\Omega\mathbf{g}(t)\cdot{\bu}^{j}(t)\,dx\qquad \forall\ t\in(0,T).
\end{split}
\end{equation}
Using (\ref{est1:eq-u:j}) instead, we can see estimates (\ref{est2:eq:u:T})-(\ref{est4:eq-u:jl}) also hold here, with
${\bu}^{j}$ in the place of ${\bu}^{j,l}$.

On the other hand, proceeding as in~\cite{O:2024a}, using assumption \eqref{g:V'}, \eqref{nuT:n:bd}, \eqref{Phi}-\eqref{Phi:d},
estimate (\ref{est2:eq:u:T}) with ${\bu}^{j}$ and $k^j$ in the places of ${\bu}^{j,l}$ and $k^{j,l}$, and  \eqref{weak-form-u:n:j}, we can show that
\begin{equation}\label{est:dt:u:j}
\int_{0}^{T}\|\pt\bu^j(t)\|_{\mathbf{V}^{s'}}^rdt \leq C(n),
\qquad r:=\min\left\{2,\frac{r_u}{2},\alpha'\right\},
\end{equation}
for some positive constant $C$, and where  $\mathbf{V}^{s'}$ denotes the dual space of $\mathbf{V}^{s}$.
Note that if $r_u=\alpha$, then $\alpha>2$ and consequently $r>1$.

We are now going to obtain estimates for $k^j$ and $\nabla k^j$ that are independent of $j$ (and of $n$).
By linearity and continuity, we can infer from (\ref{weak-form-k:tr:j}) that
\begin{equation}\label{weak-form-k:tr:j:w}
\begin{split}
& \frac{d}{dt}\int_{\Omega} k^{j}(t)w\,dx - \int_{\Omega}k^{j}(t){\bu}^{j}(t)\cdot\nabla w\,dx
  +\int_{\Omega}\nu_D^{(n)}(k^{j}(t))\nabla k^{j}(t)\cdot\nabla w\,dx
+ \int_{\Omega}\varepsilon(k^{j}(t))w\,dx \\
& = \int_{\Omega}\nut^{(n)}(k^{j}(t))|\mathbf{D}({\bu}^{j}(t))|^2w\,dx + \int_{\Omega}\nu_P^{(n)}(k^{j}(t))|{\bu}^{j}(t)|^\beta w\,dx
\end{split}
\end{equation}
holds for all $t\in(0,T)$ and all $w\in W^{1,2}_0(\Omega)$.
Note that the reasoning used to obtain (\ref{est4_0:eq-k})-(\ref{est4:eq-k:t}) is no longer valid here, because the estimates there depend on $j$ (and $n$).
The estimate established in the first next lemma results from testing (\ref{weak-form-k:tr:j:w}) with $w=\mathcal{T}_1(k^j)$, where
$\mathcal{T}_1(k^j)$ is the truncation of $k^j$ defined in (\ref{trunc:T}) for $n=1$.

\begin{lemma}\label{lem:est:kj}
Assume the identity (\ref{weak-form-k:tr:j:w}) is valid for ${\bu}^{j}$ and $k^j$ in the above conditions.
If (\ref{Hyp:theta:gamma}) holds,
then there exists an independent of $j$ (and $n$) positive constant $K$ such that
\begin{equation}\label{est:kj:infty:sup:1}
\sup_{t\in[0,T]}\|k^{j}(t)\|_1 + \int_0^T\|k^j(t)\|_{\vartheta+1}^{\vartheta+1}dt \leq K.
\end{equation}
\end{lemma}

\begin{proof}
Taking $w=\mathcal{T}_1(k^j)$ in (\ref{weak-form-k:tr:j:w}), we get
\begin{equation}\label{weak-form-k:tr:j:w=T1}
\begin{split}
& \frac{d}{dt}\|\mathcal{H}_1(k^j(t))\|_1  -\int_\Omega{\bu}^{j}(t)\cdot\nabla\mathcal{H}_1(k^j(t))\,dx \\
& +
\int_{\Omega}\nu_D^{(n)}(k^{j}(t))\nabla k^j(t)\cdot\nabla\big(\mathcal{T}_1(k^j(t))\big)\,dx + \int_\Omega\varepsilon(k^j(t))\mathcal{T}_1(k^j(t))\,dx = \\
&
\int_{\Omega}\nut^{(n)}(k^{j}(t))|\mathbf{D}({\bu}^{j}(t))|^2\mathcal{H}_1(k^j(t))\,dx + \int_{\Omega}\nu_P^{(n)}(k^{j}(t))|{\bu}^{j}(t)|^\beta\mathcal{H}_1(k^j(t))\,dx,
\end{split}
\end{equation}
where $\mathcal{H}_1$ is the primitive function of $\mathcal{T}_1$,
\begin{equation}\label{funct:H1}
\mathcal{H}_1(k):=\int_0^k\mathcal{T}_1(s)\,ds.
\end{equation}
Proceeding as in~\cite{O:2024a} (see also~\cite{BLM:2011}), in particular using assumptions (\ref{f:P-turb}) and (\ref{f:dissip-turb}), together with (\ref{eq:0:uj:kj})$_2$ and (\ref{kj>=0}), and estimate (\ref{est2:eq:u:T}) with ${\bu}^{j}$ and $k^{j}$ in the places of  ${\bu}^{j,l}$ and $k^{j,l}$,
we obtain
\begin{equation}\label{est2:eq-k:a:sup}
\sup_{t\in[0,T]}\|\mathcal{H}_1(k^j(t))\|_1 + c_\varepsilon\int_0^T\|k^j(t)\|_{\vartheta+1}^{\vartheta+1}dt \leq \|\mathcal{H}_1(k^j_0)\|_1 + C + C_P\int_0^T\int_\Omega |{\bu}^j|^\beta (1+|k^j|)^\gamma dxdt\,,
\end{equation}
for some positive constant $C$.
On the other hand, by the definition of the function $\mathcal{H}_1$, it can be easily proved the existence of two absolute positive constants $C_1$ and $C_2$ such that
\begin{equation*}
k-C_1\leq \mathcal{H}_1(k)\leq C_2k\qquad \forall\ k\in\mathds{R}^+_0.
\end{equation*}
Using this fact, together with the Young inequality, we get from  (\ref{est2:eq-k:a:sup})
\begin{equation*}
\begin{split}
& \sup_{t\in[0,T]}\|k^{j}(t)\|_1 + c_\varepsilon\int_0^T\|k^j(t)\|_{\vartheta+1}^{\vartheta+1}dt \leq \\
& C_1\|k_0^j\|_1 + C_2+C_3'\left(\int_0^T\|{\bu}^j(t)\|^\beta_\beta dt + \int_0^T\int_\Omega |{\bu}^j|^\beta |k^j|^\gamma dxdt \right) \leq \\
& C_1\|k_0^j\|_1 + C_2+C_3\left(\int_0^T\|{\bu}^j(t)\|^\beta_\beta dt + \int_0^T\|{\bu}^j(t)\|^{r_u}_{r_u} dt +
\int_0^T\|k^j(t)\|_{\frac{\gamma r_u}{r_u-\beta}}^{\frac{\gamma r_u}{r_u-\beta}}dt \right),\quad \beta<r_u,
\end{split}
\end{equation*}
for some positive constants $C_1$, $C_2$, $C_3'$ and $C_3$.
Observe that, analogously to (\ref{est4:eq-u:jl}), we can also use parabolic interpolation to show that
\begin{equation}\label{est:hj:beta}
\int_0^T\|{\bu}^{j}(t)\|_{\beta}^{\beta}dt\leq C,\quad \beta\leq r_u,
\end{equation}
for the positive constant $C$.
The estimate (\ref{est4:eq-u:jl}) with ${\bu}^{j}$ in the place of ${\bu}^{j,l}$, together with (\ref{sc:rn00}), (\ref{sc:rn0})$_2$ and (\ref{est:hj:beta}), imply
\begin{equation}\label{est:hj:infty:sup}
\sup_{t\in[0,T]}\|k^{j}(t)\|_1 + c_\varepsilon\int_0^T\|k^j(t)\|_{\vartheta+1}^{\vartheta+1}dt \leq C_1 +
C_2\int_0^T\|k^j(t)\|_{\frac{\gamma r_u}{r_u-\beta}}^{\frac{\gamma r_u}{r_u-\beta}}dt,
\end{equation}
for some positive constants $C_1$ and  $C_2$.
Now, in view of assumption (\ref{Hyp:theta:gamma}),
\begin{equation}\label{est:gamma:1}
\frac{\gamma r_u}{r_u-\beta} \leq \vartheta+1,
\end{equation}
and so
we can use the Hölder and Young inequalities to show that (\ref{est:kj:infty:sup:1}) follows from (\ref{est:hj:infty:sup}).
\end{proof}

To obtain an estimate for $\nabla k^j$, we consider the following special test function in the spirit of \cite{R:1991} (see also \cite{BP:1984,BG:1992}),
\begin{equation*}
  \upsilon(k^j):=1-\frac{1}{(1+k^j)^\delta},\quad \mbox{with}\ \ 0<\delta \ll 1.
\end{equation*}
Observe that $\upsilon(k^j)$ satisfies to
\begin{equation}\label{grad-test-kj}
 0\leq \upsilon(k^j)\leq 1,\quad \nabla\upsilon(k^j)=\delta\frac{\nabla k^j}{(1+k^j)^{\delta+1}}
\end{equation}
and therefore $\upsilon(k^j)\in L^2(0,T;H^1_0(\Omega))$. %

\begin{lemma}\label{lem:est:grad:kj}
Assume we are in the conditions of Lemma~\ref{lem:est:kj}.
\begin{enumerate}[label=(\arabic*),leftmargin=*,topsep=0pt]
\item If $\zeta>1$, then there exists an independent of $j$ (and $n$) positive constant $K$ such that
\begin{equation}\label{est:grad:kj:L2:0}
\int_0^T\|\nabla k^j(t)\|_2^2dt \leq \frac{K}{\delta}\qquad \forall\ \delta>0\ \ \mbox{small}.
\end{equation}
\item If $0\leq \zeta\leq 1$, then there exist independent of $j$ (and $n$) positive constants $K_1$ and $K_2$ such that
\begin{equation}\label{est:grad:kj:Lq:0}
\int_0^T\|\nabla k^j(t)\|_q^qdt \leq K_1+\frac{K_2}{\delta}\quad \forall\ \delta>0\ \ \mbox{small},
\quad q<1+\frac{d\zeta + 1}{d+1}.
\end{equation}
\end{enumerate}
\end{lemma}

\begin{proof}
Taking $w=\upsilon(k^j(t))$ in (\ref{weak-form-k:tr:j:w}) so that, after integrating the resulting equation between $0$ and $t\in(0,T)$, using (\ref{eq:0:uj:kj})$_2$, and taking the supreme in the interval $[0,T]$, we get
\begin{equation}\label{syst-ODEs-kj:trunc}
\begin{split}
   &
   \sup_{t\in[0,T]}\|\Upsilon(k^j(t))\|_1
   +\int_0^T\int_\Omega {\bu}^{j}\cdot\nabla\Upsilon(k^j)\,dxdt
   +\delta\int_0^T\int_\Omega \nu_D^{(n)}(k^j)\frac{|\nabla k^j|^2}{(1+k^j)^{\delta+1}}\,dxdt \\
   & +\int_0^T\int_\Omega \varepsilon(k^j)\upsilon(k^j)\,dxdt \\
   &
   =   \|\Upsilon(k^j_0)\|_1 +
   \int_0^T\int_\Omega \nut^{(n)}(k^j)|\mathbf{D}({\bu}^{j})|^2\upsilon(k^j)\,dxdt
   +\int_0^T\int_\Omega \nu_P^{(n)}(k^j)|{\bu}^{j}|^\beta\,\upsilon(k^j)\,dxdt,
\end{split}
\end{equation}
where $\Upsilon(k)$ is the following primitive function of $\upsilon(k)$,
\begin{equation}\label{Primitive:Phi}
\Upsilon(k):=\int_0^k\upsilon(s)\,ds.
\end{equation}
The second l.h.s. term of (\ref{syst-ODEs-kj:trunc}) vanishes, since ${\bu}^{j}$ is divergence free and has zero trace on the boundary $\partial\Omega$.
The fourth l.h.s. term is nonnegative due to (\ref{f:diss-turb2:j}) and  (\ref{grad-test-kj})$_1$.
In addition, since $\upsilon(k^j)\leq 1$, we obtain from (\ref{syst-ODEs-kj:trunc}),
\begin{equation}\label{est1:grad-k}
   \delta\int_0^T\int_\Omega \nu_D^{(n)}(k^j)\frac{|\nabla k^j|^2}{(1+k^j)^{\delta+1}}\,dxdt \leq \|\Upsilon(k^j_0)\|_1 +  C
   + C_P\int_0^T\int_\Omega |{\bu}^j|^\beta (1+k^j)^\gamma dxdt,
\end{equation}
for the positive constant $C$ from the counterpart estimate of (\ref{est2:eq:u:T}) that we also have used, as well as assumption (\ref{f:P-turb}), and (\ref{trunc:T})-(\ref{comp:n}).
Using assumption (\ref{f:diff-turb}) together with (\ref{trunc:T})-(\ref{comp:n}), (\ref{sc:rn00}) and (\ref{sc:rn0})$_2$,  (\ref{Primitive:Phi}), and with the fact that $|\Upsilon(k)|\leq |k|$ for all $k\in\mathds{R}$, one gets from (\ref{est1:grad-k})
\begin{equation*}
   \delta c_D\int_0^T\int_\Omega \frac{|\nabla k^j|^2}{(1+k^j)^{\delta+1-\zeta}}\,dxdt \leq \|k_0\|_1 +  C +
   C_P\int_0^T\int_\Omega |{\bu}^j|^\beta (1+k^j)^\gamma dxdt.
\end{equation*}
Proceeding as we did for (\ref{est:hj:infty:sup}), we get
\begin{equation}\label{est:hj:delta:1}
\delta c_D\int_0^T\int_\Omega \frac{|\nabla k^j|^2}{(1+k^j)^{\delta+1-\zeta}}\,dxdt \leq
C_1 +
C_2\int_0^T\|k^j(t)\|_{\frac{\gamma r_u}{r_u-\beta}}^{\frac{\gamma r_u}{r_u-\beta}}dt,\quad \beta<r_u
\end{equation}
for some positive constants $C_1$ and $C_2$.
Again, in view of assumption (\ref{Hyp:theta:gamma}), (\ref{est:gamma:1}) holds true.
Thus, reasoning for the r.h.s. term of (\ref{est:hj:delta:1}) as we did for the corresponding term of (\ref{est:hj:infty:sup}), and then using (\ref{est:kj:infty:sup:1}), we obtain
\begin{equation}\label{est:grad:kj:L2:1}
\delta c_D\int_0^T\int_\Omega \frac{|\nabla k^j|^2}{(1+k^j)^{\delta+1-\zeta}}\,dxdt \leq C
\end{equation}
for some positive constant $C$.

(1) If $\zeta>1$, we observe that this assumption implies $\zeta\geq\delta+1$ for $\delta>0$ small enough, and this, in turn, implies $(1+k^j)^{\zeta-(\delta+1)}\geq 1$.
Hence, we can easily show that (\ref{est:grad:kj:L2:1}) implies (\ref{est:grad:kj:L2:0}).

(2)
If $0\leq \zeta\leq 1$, then $\delta+1-\zeta>0$, and we can apply the Young inequality, together with estimate (\ref{est:grad:kj:L2:1}), so that
\begin{equation}\label{est:grad:kj:q} 
\begin{split}
\int_0^T\|\nabla k^j(t)\|^q_qdt = &\int_0^T\int_\Omega \frac{|\nabla k^j|^q}{(1+k^j)^{(\delta+1-\zeta)\frac{q}{2}}}(1+k^j)^{(\delta+1-\zeta)\frac{q}{2}}dxdt,
\quad q< 2
\\
\leq &
\delta c_D\int_0^T\int_\Omega \frac{|\nabla k^j|^2}{(1+k^j)^{(\delta+1-\zeta)}}dxdt + C_1' +
C_2\int_0^T\int_\Omega  |k^j|^{(\delta+1-\zeta)\frac{q}{2-q}}dxdt \\
\leq &
C_1 + C_2\int_0^T\|k^j(t)\|_{(\delta+1-\zeta)\frac{q}{2-q}}^{(\delta+1-\zeta)\frac{q}{2-q}}dt
\end{split}
\end{equation}
for some positive constants $C_1'$, $C_1$ and $C_2$.

Let us now consider the following function related with the weight $(1+k^j)^{\zeta-\delta-1}$ of the first r.h.s. integral in (\ref{est:grad:kj:q}),
\begin{equation}\label{f:Lambda}
\Lambda(k):=\int_0^k (1+s)^{\frac{\zeta-\delta-1}{2}}ds.
\end{equation}
It can be easily proved the existence of two positive constants $C_1$ and $C_2$ such that
\begin{equation}\label{Bdn:f:Lambda}
C_1\left[(1+k)^{\frac{\zeta-\delta+1}{2}}-1\right]\leq \Lambda(k) \leq C_2(1+k)^{\frac{\zeta-\delta+1}{2}}\qquad \forall\ k\in\mathds{R}^+_0.
\end{equation}
Moreover, using (\ref{f:Lambda}) and the Sobolev inequality, together with estimate (\ref{est:grad:kj:L2:1}),
there holds
\begin{equation}\label{Bdn:norm:Lambda}
\begin{split}
\int_0^T\left(\|\Lambda(k^j(t))\|_2^2+\|\nabla\Lambda(k^j(t))\|_2^2\right)dt & \leq
C_1\int_0^T\left\|\nabla\big(\Lambda(k^j(t))\big)\right\|_2^2dt  \\
& =
C_1 \int_0^T\int_\Omega \frac{|\nabla k^j|^2}{(1+k^j)^{\delta+1-\zeta}}\,dxdt
\leq
\frac{C_2}{\delta}
\end{split}
\end{equation}
for some positive constants $C_1$ and $C_2$.

Then, we use interpolation so that
\begin{equation}\label{interp:kj}
\|k^j\|_{(\delta+1-\zeta)\frac{q}{2-q}}\leq \|k^j\|_1^{1-\lambda} \|k^j\|_{(\zeta-\delta+1)\frac{\sigma}{2}}^\lambda,\qquad
\lambda=\frac{\sigma}{q}\frac{(\zeta-\delta+1)[(\delta+2-\zeta)q-2]}{(\delta+1-\zeta)[(\zeta-\delta+1)\sigma-2]},
\end{equation}
where $\sigma$ denotes the Sobolev conjugate of $2$.
Next, we use (\ref{Bdn:f:Lambda}), the Sobolev inequality and (\ref{Bdn:norm:Lambda}) so that
\begin{equation}\label{est5:grad-k}
\begin{split}
\int_0^T\|k^j(t)\|_{(\zeta-\delta+1)\frac{\sigma}{2}}^{\zeta-\delta+1}dt< &
\int_0^T\left(\int_\Omega (1+k^j)^{\frac{\zeta-\delta+1}{2}\sigma}dx\right)^{\frac{2}{\sigma}}dt \leq
\frac{1}{C_1}\int_0^T\|\Lambda(k(t))+1\|_\sigma^2\,dt \\
\leq & C_2\int_0^T\left\|\nabla\big(\Lambda(k^j(t))\big)\right\|_2^2dt \leq\frac{C_3}{\delta},
\end{split}
\end{equation}
where $C_1$ is the corresponding constant from (\ref{Bdn:f:Lambda}), and $C_2$ and $C_3$ are two other positive constants.

We raise (\ref{interp:kj}) to the power $(\delta+1-\zeta)\frac{q}{2-q}$, then we integrate the resulting inequality
between $0$ and $t\in(0,T)$ and take the supreme in $[0,T]$, which, in view of (\ref{est:kj:infty:sup:1}), implies
\begin{equation}\label{est6:grad-k}
\begin{split}
\int_0^T\|k^j(t)\|_{(\delta+1-\zeta)\frac{q}{2-q}}^{(\delta+1-\zeta)\frac{q}{2-q}}\,dt\leq &
\sup_{t\in[0,T]}\|k^j(t)\|_1^{(1-\lambda)(\delta+1-\zeta)\frac{q}{2-q}}
\int_0^T \|k^j(t)\|_{(\zeta-\delta+1)\frac{\sigma}{2}}^{\lambda(\delta+1-\zeta)\frac{q}{2-q}} dt \\
\leq & C \int_0^T \|k^j(t)\|_{(\zeta-\delta+1)\frac{\sigma}{2}}^{\lambda(\delta+1-\zeta)\frac{q}{2-q}} dt
\end{split}
\end{equation}
for some positive constant $C$.
In order to use (\ref{est5:grad-k}), we choose $q$ so that
\begin{equation}\label{q1:est:grad:k}
\lambda(\delta+1-\zeta)\frac{q}{2-q}=\zeta-\delta+1
\Leftrightarrow
q=\frac{2\sigma(\zeta-\delta+2)-4}{3\sigma-2}
\Leftrightarrow
\left\{
\begin{array}{ll}
q=\frac{d\zeta-d\delta + d+2}{d+1}, & d\not=2
 \\
q<\frac{2\zeta-2\delta+4}{3}, & d=2.
\end{array}
\right.
\end{equation}
Hence, combining (\ref{est5:grad-k}) with (\ref{est6:grad-k}), we have
\begin{equation}\label{est7:grad-k}
\int_0^T\|k^j(t)\|_{(\delta+1-\zeta)\frac{q}{2-q}}^{(\delta+1-\zeta)\frac{q}{2-q}}\,dt\leq
C_1 \int_0^T \|k^j(t)\|_{(\zeta-\delta+1)\frac{\sigma}{2}}^{\zeta-\delta+1} dt\leq
\frac{C_2}{\delta}
\end{equation}
for some positive constants $C_1$ and $C_2$.

Note that, in view of (\ref{est:grad:kj:q}) and (\ref{q1:est:grad:k}),
\begin{equation*}
q< 2 \Leftrightarrow \zeta-\delta+1<2
\end{equation*}
which is true for a sufficiently small $\delta>0$, and because $0\leq \zeta\leq 1$.
Plugging (\ref{est7:grad-k}) into (\ref{est:grad:kj:q}), and observing the requirements for the exponent $q$ declared at (\ref{interp:kj}) and (\ref{q1:est:grad:k}), we prove that (\ref{est:grad:kj:Lq:0}) holds true.
\end{proof}


Combining Lemmas~\ref{lem:est:kj} and \ref{lem:est:grad:kj}, we can now establish the following result.
Note that estimate (\ref{est4:eq-k:t}) is not an alternative here, because it depends on $j$ (and $n$).

\begin{lemma}\label{lem:kj:any:q}
Assume we are in the conditions of Lemmas~\ref{lem:est:kj}-\ref{lem:est:grad:kj}.
Then, there exists an independent of $j$ (and $n$) positive constant $K$ such that
\begin{equation}\label{est:kj:rho:q<rk}
\int_0^T\|k^j(t)\|^{r}_{r} dt \leq
\frac{K}{\delta} \quad \forall\ \delta>0,\qquad r<r_k,\quad\mbox{for $r_k$ given in (\ref{ru:rk}).}
\end{equation}
\end{lemma}

\begin{proof}
For any $\zeta\geq 0$ and $\delta:0<\delta<<1$, we can use interpolation so that
\begin{equation}\label{int:rho:any:q}
\|k^j(t)\|_r\leq \|k^j(t)\|_1^{1-\lambda}\|k^j(t)\|_{(\zeta-\delta+1)\frac{\sigma}{2}}^\lambda,\quad
\lambda=\frac{(r-1)(\zeta-\delta+1)\sigma}{r[\sigma(\zeta-\delta+1)-2]},
\end{equation}
where $\sigma$ is the Sobolev conjugate of $2$.
Raising (\ref{int:rho:any:q}) to the power $r$, and then, in order to use (\ref{est5:grad-k}) again, we choose $r$ so that
\begin{equation}\label{rho:any:q}
r\lambda=\zeta-\delta+1 \Leftrightarrow r=\zeta-\delta+2-\frac{2}{\sigma}
\Rightarrow r:
\left\{
\begin{array}{ll}
=\zeta-\delta+1+\frac{2}{d} , & d\not=2, \\
<\zeta-\delta+2, & d=2.
\end{array}
\right.
\end{equation}
Next, we integrate the resulting inequality between $0$ and $t\in[0,T]$, to get, after the application of estimates (\ref{est:kj:infty:sup:1}) and (\ref{est5:grad-k}) in the final part,
\begin{equation}\label{int:0T:rho:any:q}
\int_0^T\|k^j(t)\|_r^r dt\leq \sup_{t\in[0,T]}\|k^j(t)\|_1^{\frac{\sigma-2}{\sigma}}\int_0^T\|k^j(t)\|_{(\zeta-\delta+1)\frac{\sigma}{2}}^{\zeta-\delta+1} \leq\frac{C}{\delta},
\end{equation}
where $C$ is a positive constant.
Hence, (\ref{est:kj:rho:q<rk}), in the case of $r<r_k=\zeta+1+\frac{2}{d}$, is now a direct consequence of (\ref{rho:any:q}) and (\ref{int:0T:rho:any:q}).
\end{proof}

\begin{remark}\label{Rem3:Proof:Exist}
Note that, due to (\ref{est:kj:infty:sup:1}), (\ref{est:kj:rho:q<rk}) does also hold true for $r=\vartheta+1$, but, in view of assumption (\ref{eq:Cond2}), this estimate is worse.
On the other hand, if $d=3$, we obtain $r<r_k=\zeta+\frac{5}{3}$ in (\ref{est:kj:rho:q<rk}), as in~\cite{BLM:2011}, and if $d=2$, we get $r<r_k=\zeta+2$.
\end{remark}

In order to obtain other estimates that neither depend on $j$ nor on $n$, we can also proceed as we did for (\ref{est:dt:k:jl}), but now using the identity (\ref{weak-form-k:tr:j:w}) instead.
We first note that, by combining the Hölder inequality with estimate (\ref{est2:eq:u:T}), that still holds with ${\bu}^{j}$ and $k^{j}$ in the places of ${\bu}^{j,l}$ and $k^{j,l}$, one has
\begin{equation}\label{est:indep:jn:3:1}
\int_0^T\left\|\nut^{(n)}(k^{j}(t))|\mathbf{D}({\bu}^{j}(t))|^2\right\|_{W^{-1,\varrho}(\Omega)}^{r}dt\leq C,\qquad 1<\varrho<\varrho_1:=\frac{d}{d-1},\quad r=1,
\end{equation}
for some positive constant $C$.

Besides estimates of Lemmas~\ref{lem:est:kj}-\ref{lem:kj:any:q}, we also need to obtain independent of $j$ (and $n$) estimates for the turbulent diffusion term, as well as for the terms of turbulence transport and turbulence production.

We start by estimating the turbulent diffusion term.
\begin{lemma}\label{lem:est:nuD:kj}
Assume we are in the conditions of Lemmas~\ref{lem:est:kj}-\ref{lem:kj:any:q}.
Then, there exists an independent of $j$ (and $n$) positive constant $K$ such that
\begin{equation}\label{est:nD:grad:kj}
\int_0^T\left\|\nu_D^{(n)}(k^{j}(t))\nabla k^j(t)\right\|_\varrho^\varrho dt \leq\frac{K}{\delta}\quad \forall\ \delta>0\ \ \mbox{small},
\quad \varrho<\varrho_2:=\frac{d\zeta+d+2}{d\zeta+d+1}.
\end{equation}
\end{lemma}

\begin{proof}
By using assumption (\ref{f:diff-turb}), together with (\ref{trunc:T})-(\ref{comp:n}), the Hölder inequality and estimate (\ref{est:grad:kj:L2:1}), we can show that
\begin{equation}\label{est4:k':q1}
\begin{split}
&
\int_0^T\int_\Omega \big|\nu_D^{(n)}(k^{j})\nabla k^{j}\big|^\varrho dxdt \leq
C_D^\varrho\int_0^T\int_\Omega (1+k^{j})^{\zeta \varrho}|\nabla k^{j}|^\varrho dxdt \\
& = C_D^\varrho
\int_0^T\int_\Omega (1+k^{j})^{(\delta+1+\zeta)\frac{\varrho}{2}}\frac{|\nabla k^{j}|^\varrho}{(1+k^{j})^{(\delta+1-\zeta)\frac{\varrho}{2}}}dxdt \\
& \leq
C_D^\varrho\left(\int_0^T\int_\Omega \frac{|\nabla k^{j}|^2}{(1+k^{j})^{\delta+1-\zeta}}dxdt\right)^{\frac{\varrho}{2}}
\left(\int_0^T\int_\Omega (1+k^{j})^{(\delta+1+\zeta)\frac{\varrho}{2-\varrho}}dxdt\right)^{\frac{2-\varrho}{\varrho}},\quad \varrho<2, \\
& \leq
\frac{C}{\delta}\left[1+\left(\int_{0}^T\|k^{j}(t)\|_{(\zeta+\delta+1)\frac{\varrho}{2-\varrho}}^{(\zeta+\delta+1)\frac{\varrho}{2-\varrho}}dt\right)^{\frac{2-\varrho}{2}}\right].
\end{split}
\end{equation}
for some positive constant $C$.
Having in mind estimate (\ref{int:0T:rho:any:q}), with $r$ given there by (\ref{rho:any:q}), we choose $\varrho$ so that
\begin{equation}\label{varrho:1}
(\zeta+\delta+1)\frac{\varrho}{2-\varrho}=\zeta-\delta+2-\frac{2}{\sigma}
\Leftrightarrow
\varrho=\frac{2(\zeta+2)\sigma-4}{(2\zeta+3)\sigma-2}-\frac{2\sigma}{(2\zeta+3)\sigma-2}\delta,
\end{equation}
where $\sigma$ denotes the Sobolev conjugate of $2$.
Since $0<\delta<<1$, we have
\begin{equation}\label{varrho<zeta:delta:d}
\varrho<\frac{2(\zeta+2)\sigma-4}{(2\zeta+3)\sigma-2}
\Leftrightarrow
\left\{
\begin{array}{ll}
\varrho<\frac{d\zeta +d+2}{d\zeta+d+1}, & d\not=2
 \\
\varrho<\frac{2\zeta+4}{2\zeta+3}, & d=2.
\end{array}
\right.
\end{equation}
As a consequence of (\ref{est:kj:rho:q<rk}), we can readily see that (\ref{est4:k':q1})  and (\ref{varrho<zeta:delta:d}) imply (\ref{est:nD:grad:kj}).
\end{proof}

\begin{remark}\label{Rem4:Proof:Exist}
Note that any $\varrho$ in the conditions of (\ref{varrho<zeta:delta:d}) satisfies also to $\varrho<2$, as required by (\ref{est4:k':q1}).
On the other hand, if $d=3$, we obtain $\varrho<\frac{3\zeta+5}{3\zeta+4}$ in (\ref{est:nD:grad:kj}), as in~\cite{BLM:2011}, and if $d=2$, we get $\varrho<\frac{2\zeta+4}{2\zeta+3}$.
\end{remark}

From (\ref{est:nD:grad:kj}), one immediately has
\begin{equation}\label{est:indep:jn:3:3}
\begin{split}
& \int_0^T\left\|\operatorname{div}(\nu_D^{(n)}(k^{j}(t))\nabla k^j(t))\right\|_{W^{-1,\varrho}(\Omega)}^{\varrho}dt \leq
 \int_0^T\left\|\nu_D^{(n)}(k^{j}(t))\nabla k^j(t)\right\|_{\varrho}^{\varrho}dt \leq \frac{K}{\delta},
\end{split}
\end{equation}
for $\varrho$ satisfying (\ref{est:nD:grad:kj}) (see also (\ref{varrho<zeta:delta:d})).

Next, we estimate the term of turbulence transport.

\begin{lemma}\label{lem:est:kj:uj}
Assume we are in the conditions of Lemmas~\ref{lem:est:kj}-\ref{lem:est:nuD:kj}.
Then, there exists an independent of $j$ (and $n$) positive constant $K$ such that
\begin{equation}\label{est:kj:uj}
\int_0^T\left\|k^j(t){\bu}^j(t)\right\|_\varrho^\varrho dt \leq\frac{K}{\delta}\quad \forall\ \delta>0\ \ \mbox{small}
\end{equation}
for
\begin{equation}\label{est:kj:uj:rho}
\varrho<\varrho_3:=\max\left\{\frac{2(d\zeta+d+2)(d+2)}{d(d\zeta+3d+6)},\frac{\alpha(d\zeta+d+2)}{d\zeta+d\alpha+d+2}\right\}.
\end{equation}
\end{lemma}

\begin{proof}
Using the Hölder inequality, one has
\begin{equation}\label{est2:ku:q1}
\int_0^T \|k^j(t){\bu}^{j}(t)\|_{\varrho}^{\varrho}dt
\leq
\left(\int_0^T\|k^j(t)\|_q^q dt\right)^{\frac{\varrho}{q}}\left(\int_0^T\|{\bu}^{j}(t)\|_{r_u}^{r_u}dt\right)^{\frac{\varrho}{r_u}},
\end{equation}
where, for $q=r$, and $r$ given in (\ref{rho:any:q}),
\begin{equation*}
\begin{split}
\frac{1}{\varrho}=\frac{1}{r_u}+\frac{1}{q} & \Leftrightarrow
\varrho=\frac{r_u[\sigma(\zeta-\delta+2)-2]}{\sigma(\zeta-\delta+2)-2+r_u\sigma} \\
&
\Leftrightarrow
\varrho=
\frac{r_u[\sigma(\zeta+2)-2]}{\sigma(\zeta+2)-2+r_u\sigma}-\tau\delta,\quad \tau:=
\frac{r_u^2\sigma^2}{[\sigma(\zeta-\delta+2)-2+r_u\sigma][\sigma(\zeta+2)-2+r_u\sigma]},
\end{split}
\end{equation*}
and where $r_u$ is given by (\ref{ru:rk}).
Recall that $\sigma$ denotes the Sobolev conjugate of $2$.
Since $\zeta\geq 0$, $0<\delta<<1$ and $\sigma>2$, we have $\tau>0$, which implies that $\tau\delta$ is very small as well.
Hence,
\begin{equation}\label{varsigma1}
\varrho<\varrho_3:=
\left\{
\begin{array}{ll}
\frac{2(d\zeta+d+2)(d+2)}{d(d\zeta+3d+6)} & \mbox{if}\ \ r_u=\frac{2(d+2)}{d},
 \\
\frac{\alpha(d\zeta+d+2)}{d\zeta+d\alpha+d+2}, & \mbox{if}\ \ r_u=\alpha
\end{array}
\right.
\end{equation}
Note that $\varrho_3> 1$ in any case.
Plugging (\ref{est4:eq-u:jl}) and (\ref{est:kj:rho:q<rk}), the first with ${\bu}^{j}$ in the place of ${\bu}^{j,l}$, into
(\ref{est2:ku:q1}), we prove (\ref{est:kj:uj}).
\end{proof}

\begin{remark}\label{Rem5:Proof:Exist}
Estimate (\ref{est:kj:uj}) already gives us a condition depending on the power-law index characterizing the Darcy-Forchheimer drag forces.
In particular, the values of interest of $\varrho_3$ in (\ref{varsigma1}), from the point of view of physics, are
\begin{equation*}
\varrho_3=
\left\{
\begin{array}{ll}
\max\left\{\frac{10}{9}\frac{3\zeta+5}{\zeta+5},\frac{\alpha(3\zeta+5)}{3\zeta+3\alpha+5}\right\}, & d=3, \\
\max\left\{\frac{4(\zeta+2)}{\zeta+6},\frac{\alpha(\zeta+2)}{\zeta+\alpha+2}\right\}, & d=2.
\end{array}
\right.
\end{equation*}
If $r_u=\frac{2(d+2)}{d}$, then $\varrho_3=\frac{10}{9}\frac{3\zeta+5}{\zeta+5}$ if $d=3$, as in~\cite{BLM:2011}, and $\varrho_3=\frac{4(\zeta+2)}{\zeta+6}$ if $d=2$.
However, if $r_u=\alpha$, then $\varrho_3=\frac{\alpha(3\zeta+5)}{3\zeta+3\alpha+5}$ if $d=3$, and $\varrho_3=\frac{\alpha(\zeta+2)}{\zeta+\alpha+2}$ if $d=2$.
\end{remark}

Now, (\ref{est:kj:uj}) and (\ref{est:kj:uj:rho}) imply
\begin{equation}\label{est2:k':q1}
\int_0^T\|\operatorname{div}\big(k^j(t){\bu}^j(t)\big)\|_{W^{-1,\varrho}(\Omega)}^{\varrho}dt \leq
\int_0^T\|k^j(t){\bu}^j(t)\|_{\varrho}^{\varrho}dt \leq   \frac{K}{\delta}\quad \forall\ \delta>0\ \ \mbox{small},
\end{equation}
for $\varrho$ satisfying \eqref{est:kj:uj:rho}.

It last to obtain an estimate for the term of turbulence production.

\begin{lemma}
Assume we are in the conditions of Lemmas~\ref{lem:est:kj}-\ref{lem:est:kj:uj}.
Then, there exists an independent of $j$ (and $n$) positive constant $K$ such that
\begin{equation}\label{est:nuPn:kj:uj:i}
\int_0^T\left\|\nu_P^{(n)}(k^j(t))|{\bu}^j|^\beta\right\|_\varrho^\varrho dt \leq
\frac{K}{\delta} \quad \forall\ \delta>0
\end{equation}
for
\begin{equation}\label{est:nuPn:kj:uj:r}
\varrho<\varrho_4:=\max\left\{\frac{2(d\zeta+d+2)(d+2)}{d\beta(d\zeta+d+2)+ 2\gamma d(d+2)},\frac{(d\zeta +d+2)\alpha}{\beta(d\zeta +d+2)+ \gamma d \alpha}\right\}.
\end{equation}
\end{lemma}

\begin{proof}
Proceeding as we did for (\ref{est4:k':q1}), we can use assumptions (\ref{f:P-turb}), (\ref{eq:Cond2}) and (\ref{Hyp:theta:gamma}), together with (\ref{trunc:T})-(\ref{comp:n}), the Hölder and Young inequalities, and (\ref{est4:eq-u:jl}), this with ${\bu}^{j}$ in the place of ${\bu}^{j,l}$, to show that
\begin{equation}\label{est:kj:nuP:n}
\begin{split}
\int_0^T\int_\Omega \big|\nu_P^{(n)}(k^{j})|{\bu}^j|^\beta\big|^\varrho dxdt \leq &
C_P^\varrho\int_0^T\int_\Omega (1+k^{j})^{\gamma\varrho}|{\bu}^j|^{\beta\varrho} dxdt \\
\leq &
C_P^\varrho\left(\int_0^T\int_\Omega (1+k^{j})^{q}dxdt\right)^{\frac{\varrho\gamma}{q}}\left(\int_0^T\int_\Omega |{\bu}^j|^{r_u} dxdt\right)^{\frac{\varrho\beta}{r_u}}
\\
\leq &
C\left(1+\int_0^T\|k^{j}(t)\|_{q}^{q}dt\right)^{\frac{\varrho\gamma}{q}}
\end{split}
\end{equation}
for some positive constant $C$, and where
    \begin{equation}\label{eq:est:nuP:un:qk}
\begin{split}
\frac{\varrho\gamma}{q}+\frac{\varrho\beta}{r_u}=1\Leftrightarrow &
    \varrho=\frac{r_uq}{\beta q+\gamma\,r_u}<
    \frac{r_u r_k}{\beta r_k+\gamma\,r_u}=
\frac{(d\zeta +d+2)r_u}{(d\zeta +d+2)\beta+ \gamma d r_u} \\
\Leftrightarrow & \varrho<\varrho_4:=
    \left\{
    \begin{array}{l}
     \frac{2(d\zeta+d+2)(d+2)}{d\beta(d\zeta+d+2)+ 2\gamma d(d+2)}
    \ \ \mbox{if}\ \ r_u=2\frac{d+2}{d} \\
    \frac{(d\zeta +d+2)\alpha}{\beta(d\zeta +d+2)+ \gamma d \alpha}
    \ \ \mbox{if}\ \ r_u=\alpha,
    \end{array}
    \right.
\end{split}
    \end{equation}
where $r_k$ is given in (\ref{ru:rk}).
Note that assumptions (\ref{eq:Cond2}) and (\ref{Hyp:theta:gamma}) assure us that $\varrho_k>1$ in any case.
Using estimate (\ref{est:kj:rho:q<rk}), with $q=r$ and $r$ given at \eqref{rho:any:q},  we can infer from (\ref{est:kj:nuP:n}) that (\ref{est:nuPn:kj:uj:i}) holds true.
\end{proof}

\begin{remark}
In the dimensions of physics interest, we have
\begin{equation*}
\varrho_4=
\left\{
\begin{array}{ll}
\max\left\{\frac{30\zeta+50}{3\beta(3\zeta+5)+30\gamma},\frac{\alpha(3\zeta+5)}{\beta(3\zeta+5)+3\alpha\gamma}\right\}
 & \mbox{if}\ d=3, \\
\max\left\{\frac{4\zeta+8}{\beta(\zeta+2)+4\gamma},\frac{\alpha(\zeta+2)}{\beta(\zeta+2)+\alpha\gamma}\right\}, & \mbox{if}\ d=2.
\end{array}
\right.
\end{equation*}
Observe that $\varrho_4=\frac{30\zeta+50}{3\beta(3\zeta+5)+30\gamma}$ if $\alpha<\frac{10}{3}$ and $d=3$, and
$\varrho_4=\frac{4\zeta+8}{\beta(\zeta+2)+4\gamma}$ if $\alpha<4$ and $d=2$.
\end{remark}

Combining the Sobolev and Hölder inequalities with \eqref{est:nuPn:kj:uj:i}, we can show that
\begin{equation}\label{est:indep:jn:3:2}
\begin{split}
& \int_0^T\left\|\nu_P^{(n)}(k^{j}(t))|{\bu}^{j}(t)|^\beta\right\|_{W^{-1,\varrho}(\Omega)}^{\varrho}dt \leq
  C\int_0^T\left\|\nu_P^{(n)}(k^{j}(t))|{\bu}^{j}(t)|^\beta\right\|_{\varrho}^{\varrho}dt\leq
\frac{K}{\delta} \quad \forall\ \delta>0,
\end{split}
\end{equation}
for $\varrho$ satisfying \eqref{est:nuPn:kj:uj:r}, and for some positive constants $C$ and $K$.

Now, we estimate the term of turbulence dissipation by using assumption (\ref{f:dissip-turb}), together with estimate \eqref{est:kj:infty:sup:1} (or (\ref{est:kj:rho:q<rk}) -- see Remark~\ref{Rem3:Proof:Exist}), so that
\begin{equation}\label{est1:indep:jn:4}
\int_0^T\left\|\varepsilon(k^{j}(t))\right\|_{W^{-1,\varrho}(\Omega)}^rdt\leq
\frac{K}{\delta},\quad \forall\ \delta>0,\quad 1<\varrho<\frac{d}{d-1}=\varrho_1,\quad r=1,
\end{equation}
for some positive constant $K$.

Finally, combining (\ref{epsilon(kj)>=0}) with (\ref{est:indep:jn:3:1}), (\ref{est:indep:jn:3:3}), (\ref{est2:k':q1}), (\ref{est:indep:jn:3:2}) and \eqref{est1:indep:jn:4},
one has
\begin{equation}\label{est:indep:jn:3:0}
\int_0^T\left\|\partial_tk^j(t)\right\|_{W^{-1,\varrho}(\Omega)}dt \leq
K_1+\frac{K_2}{\delta}\quad \forall\ \delta>0\ \ \mbox{small},\quad \varrho<\varrho_0:=\min\left\{\varrho_1,\varrho_2,\varrho_3,\varrho_4\right\}
\end{equation}
for some positive constants $C_1$  and $C_2$, and where $\varrho_1$, $\varrho_2$, $\varrho_3$, $\varrho_4$ are defined in (\ref{est:indep:jn:3:1}), (\ref{est:nD:grad:kj}), (\ref{est:kj:uj:rho}) and (\ref{est:nuPn:kj:uj:r}).
The precise definition of $\varrho_0$ is
\begin{equation}\label{varrho:0}
\begin{split}
\varrho_0:= &
\min\Bigg\{
\frac{d}{d-1},
\frac{d\zeta+d+2}{d\zeta+d+1},
\max\left\{\frac{2(d\zeta+d+2)(d+2)}{d(d\zeta+3d+6)},\frac{\alpha(d\zeta+d+2)}{d\zeta+d\alpha+d+2}\right\}, \\
&
\max\left\{\frac{2(d\zeta+d+2)(d+2)}{d\beta(d\zeta+d+2)+ 2\gamma d(d+2)},\frac{(d\zeta +d+2)\alpha}{\beta(d\zeta +d+2)+ \gamma d \alpha}\right\}
\Bigg\}
\end{split}
\end{equation}
Note that $\varrho_1,\ \varrho_2,\ \varrho_3,\ \varrho_4>1$ and therefore $\varrho:1<\varrho<\varrho_0$ can be chosen.

\section{Passing to the limit as $j\to\infty$}

In this section, all the considered  subsequences will still be labeled by the sequence superscript $j$.
As observed in the previous section, estimates (\ref{est2:eq:u:T})-(\ref{est4:eq-u:jl}) do not depend on $j$ and therefore also hold with ${\bu}^{j}$ and $k^{j}$ in the places of ${\bu}^{j,l}$ and $k^{j,l}$.
In view of this and estimate (\ref{est:dt:u:j}), which also does not depend on $j$, we may appeal to the Banach-Alaoglu theorem so that for some subsequences
\begin{alignat}{3}
& \label{w*:conv:uj}
{\bu}^{j}\xrightharpoonup[j\to\infty]{\ast} {\bu}\ \ \mbox{in}\ \ L^\infty(0,T;\mathbf{H}), && \\
  & \label{w:conv:uj}
{\bu}^{j}\xrightharpoonup[j\to\infty]{} {\bu},\ \ \mbox{in}\ L^2(0,T;\mathbf{V})\cap L^{r_u}(0,T;L^{r_u}(\Omega)^d), && \\
  & \label{w:conv:uj'}
\partial_t{\bu}^{j}\xrightharpoonup[j\to\infty]{} \partial_t{\bu}\ \ \mbox{in}\ \ L^r(0,T;\mathbf{V}^{s'}),
\end{alignat}
for $r_u$ given in \eqref{ru:rk} and $r$ given in \eqref{est:dt:u:j}.
From (\ref{est:kj:infty:sup:1}), (\ref{est:grad:kj:L2:0})-(\ref{est:grad:kj:Lq:0}), (\ref{est:kj:rho:q<rk}) and (\ref{est:indep:jn:3:0}), we can also deduce from the Banach-Alaoglu theorem that for some subsequences
\begin{alignat}{2}
& \label{w*:conv:kj}
k^{j}\xrightharpoonup[j\to\infty]{\ast} k\ \ \mbox{in}\ \ L^\infty(0,T;\mathcal{M}(\Omega)), && \\
  & \label{w:conv:kj:W1q}
k^{j}\xrightharpoonup[j\to\infty]{} k\ \ \mbox{in}\ \ L^q(0,T;W^{1,q}_0(\Omega)),\qquad
1<q<\min\left\{2,1+\frac{d\zeta + 1}{d+1}\right\}, && \\
 & \label{w:conv:kj:rho:theta}
k^{j}\xrightharpoonup[j\to\infty]{} k\ \ \mbox{in}\ \ L^q(0,T;L^q(\Omega))\cap L^{\vartheta+1}(0,T;L^{\vartheta+1}(\Omega)),\qquad 1\leq q<r_k, \\
  & \label{w:conv:kj'}
\partial_tk^j\xrightharpoonup[j\to\infty]{} \partial_tk\ \ \mbox{in}\ \ \mathcal{M}(0,T;W^{-1,\rho}(\Omega)),\quad
1<\rho<\rho_0,
\end{alignat}
for $r_k$ given in (\ref{ru:rk}) and $\rho_0$ in \eqref{est:indep:jn:3:0}.

By using estimates (\ref{est2:eq:u:T})-(\ref{est4:eq-u:jl}), with ${\bu}^{j}$ and $k^{j}$ in the places of ${\bu}^{j,l}$ and $k^{j,l}$, together with (\ref{est:dt:u:j}) and the generalized Aubin-Lions compactness lemma (see~\cite[Corollary~6]{Simon:1987}), we have
\begin{equation}\label{str:conv:uj}
{\bu}^{j}\xrightarrow[j\to\infty]{}{\bu}\ \ \mbox{in}\ \ L^q(0,T;L^q(\Omega)^d),\qquad 1\leq q< r_u.
\end{equation}
By the same arguing, from estimates (\ref{est:kj:infty:sup:1}), (\ref{est:grad:kj:L2:0})-(\ref{est:grad:kj:Lq:0}), (\ref{est:kj:rho:q<rk}) and (\ref{est:indep:jn:3:0}), we have
\begin{equation}\label{str:conv:kj}
k^{j}\xrightarrow[j\to\infty]{}k\ \ \mbox{in}\ \ L^q(0,T;L^q(\Omega)),\qquad 1\leq q<r_k.
\end{equation}
Now, in view of (\ref{str:conv:uj}) and (\ref{str:conv:kj}) and the Riesz-Fischer theorem, there exist another subsequences such that
\begin{alignat}{2}
& \label{ae:conv:uj}
{\bu}^{j}\xrightarrow[j\to\infty]{}{\bu}\quad \mbox{a.e.\ \ in}\ \ Q_T,  && \\
  & \label{ae:conv:kj}
k^{j}\xrightarrow[j\to\infty]{}k\quad \mbox{a.e.\ \ in}\ \ Q_T. &&
\end{alignat}

Using (\ref{w:conv:uj}), (\ref{ae:conv:uj}) and (\ref{ae:conv:kj}), and reasoning as in~\cite{O:2024a} (see also \eqref{w:conv:Phi-d:jl} and \eqref{s:conv:nuT(k)Dv:jl:L2}), we can show that
\begin{alignat}{2}
& \label{w:conv:ujxuj}
\Phi_n(|\bu^j|^2)\bu^{j}\otimes\bu^{j}\xrightharpoonup[j\to\infty]{} \Phi_n(|\bu|^2)\bu\otimes\bu\ \ \mbox{in}\ \ L^{\frac{d+2}{d}}(0,T;L^{\frac{d+2}{d}}(\Omega)^{d\times d}), && \\
& \label{w:conv:f(uj)}
|\bu^{j}|^{\alpha-2}\bu^{j}\xrightharpoonup[j\to\infty]{} |\bu|^{\alpha-2}\bu
\ \ \mbox{in}\ \ L^{\alpha'}(0,T;L^{\alpha'}(\Omega)^d), && \\
& \label{w:conv:nuT(k)Dv:j}
\nut^{(n)}(k^{j})\mathbf{D}({\bu}^{j})\xrightharpoonup[j\to\infty]{}  \nut^{(n)}(k)\mathbf{D}({\bu})\ \ \mbox{in}\ \ L^2(0,T;L^{2}(\Omega)^{d\times d}), && \\
& \label{w:conv:nuT(k)Dv:j:sqrt}
\sqrt{\nut^{(n)}(k^{j})}\mathbf{D}({\bu}^{j})\xrightharpoonup[j\to\infty]{}  \sqrt{\nut^{(n)}(k)}\mathbf{D}({\bu})\ \ \mbox{in}\ \ L^2(0,T;L^{2}(\Omega)^{d\times d}). &&
\end{alignat}

The convergence results (\ref{w*:conv:uj})-(\ref{w:conv:uj'}) and  (\ref{w:conv:ujxuj})-(\ref{w:conv:nuT(k)Dv:j}) are sufficient to pass the equation (\ref{weak-form-u:n:j}) to the limit $j\to\infty$.
In view of this, and by means of linearity and continuity, we can see that, for any fixed $n$, (\ref{weak-form-u:n}) holds true for any $\bv\in \mathbf{V}\cap L^\alpha(\Omega)^d$.
By a standard procedure (see e.g.~\cite[Lemma~III.1.2]{Temam:1979}), we can invoke (\ref{w*:conv:uj})-(\ref{w:conv:uj'}) and (\ref{w:conv:nuT(k)Dv:j:sqrt}) to prove that ${\bu}\in L^\infty(0,T;\mathbf{H})$ is weakly continuous with values in $\mathbf{H}$, i.e.
${\bu}\in C_{\rm{w}}([0,T];\mathbf{H})$, and hence (\ref{problem4:n})$_1$ is meaningful.

Taking $\bv=\bu(t)$ in (\ref{weak-form-u:n}), integrating the resulting identity between $0$ and $T$, using (\ref{problem2:n}) and (\ref{problem4:n})$_1$, and arguing as we did for \eqref{est2:eq:u:T}, we have
\begin{equation}
\begin{split}\label{weak-form-u:n:sqrt}
&
\int_0^T\left\|\sqrt{\nut^{(n)}(k(t))}\mathbf{D}(\bu(t))\right\|_2^2dt
+ c_{Fo}\int_0^T\|\bu(t)\|^{\alpha}_\alpha dt = \\
&
-\frac{1}{2}\|\bu(T)\|^2_2 + \frac{1}{2}\|\bu_{n,0}\|^2_2
-c_{Da}\int_0^T\|\bu(t)\|^2_2dt +
\int_0^T\int_\Omega\bg\cdot\bu\, dxdt.
\end{split}
\end{equation}

Integrating (\ref{est1:eq-u:j}) between $0$ and $T$, and using (\ref{eq:0:uj:kj})$_1$, one has
\begin{equation*}
\begin{split}
&
\int_0^T\left\|\sqrt{\nut^{(n)}(k^j(t))}\mathbf{D}(\bu^j(t))\right\|_2^2dt +
c_{Fo}\int_0^T\|\bu^j(t)\|^{\alpha}_\alpha dt = \\
&
 -\frac{1}{2}\|\bu^j(T)\|^2_2 + \frac{1}{2}\|\bu^j_0\|^2_2
 -c_{Da}\int_0^T\|\bu^j(t)\|^2_2dt
 + \int_0^T\int_\Omega\bg\cdot\bu^j\,dx dt.
\end{split}
\end{equation*}
Now, shifting the first term in the r.h.s. to the left, letting $j\to\infty$ and using the convergence results (\ref{sc:vn0})$_2$, (\ref{w:conv:uj}) and (\ref{str:conv:uj}), we obtain
\begin{equation}\label{lim:sup:1}
\begin{split}
&
\limsup_{j\to\infty}\left(
\int_0^T\left\|\sqrt{\nut^{(n)}(k^j(t))}\mathbf{D}(\bu^j(t))\right\|_2^2dt
+
c_{Fo}\int_0^T\|\bu^j(t)\|^{\alpha}_\alpha dt\right) +
\limsup_{j\to\infty}\left(\frac{1}{2}\|\bu^j(T)\|^2_2\right)
\leq \\
&
\frac{1}{2}\|\bu_{n,0}\|^2_2 - c_{Da}\int_0^T\|\bu(t)\|^2_2dt
 + \int_0^T\int_\Omega\bg\cdot\bu\,dx dt
\end{split}
\end{equation}
Observing that, due to (\ref{str:conv:uj}) and the lower semi-continuity of the norm,
\begin{equation*}
\frac{1}{2}\|\bu(T)\|^2_2\leq\limsup_{j\to\infty}\left(\frac{1}{2}\|\bu^j(T)\|^2_2\right),
\end{equation*}
we can plug (\ref{weak-form-u:n:sqrt}) into (\ref{lim:sup:1}), so that
\begin{equation}\label{lim:sup:2}
\begin{split}
&
\limsup_{j\to\infty}\left(
\int_0^T\left\|\sqrt{\nut^{(n)}(k^j(t))}\mathbf{D}(\bu^j(t))\right\|_2^2dt
+
c_{Fo}\int_0^T\|\bu^j(t)\|^{\alpha}_\alpha dt\right) \leq \\
&
\int_0^T\left\|\sqrt{\nut^{(n)}(k(t))}\mathbf{D}(\bu(t))\right\|_2^2dt
+
c_{Fo}\int_0^T\|\bu(t)\|^{\alpha}_\alpha dt.
\end{split}
\end{equation}

On the other hand, by (\ref{w:conv:f(uj)}), (\ref{w:conv:nuT(k)Dv:j:sqrt}) and the lower semi-continuity of the norms, there holds
\begin{equation}\label{lim:sup:3}
\begin{split}
&
\int_0^T\left\|\sqrt{\nut^{(n)}(k(t))}\mathbf{D}({\bu}(t))\right\|_2^2dt
+
c_{Fo}\int_0^T\|{\bu}(t)\|^{\alpha}_\alpha dt \leq \\
&
\liminf_{j\to\infty}\left(\int_0^T\left\|\sqrt{\nut^{(n)}(k^j(t))}\mathbf{D}({\bu}^j(t))\right\|_2^2dt
+
c_{Fo}\int_0^T\|{\bu}^j(t)\|^{\alpha}_\alpha dt\right).
\end{split}
\end{equation}

Then, combining (\ref{lim:sup:2}) with (\ref{lim:sup:3}), we obtain
\begin{equation*}
\nut^{(n)}(k^j)|\mathbf{D}({\bu}^j)|^2 + c_{Fo}|{\bu}^j|^{\alpha}
\xrightarrow[j\to\infty]{}\nut^{(n)}(k)|\mathbf{D}({\bu})|^2 + c_{Fo}|{\bu}|^{\alpha}\ \ \mbox{in}\ \ L^1(0,T;L^1(\Omega)),
\end{equation*}
which implies, by the uniqueness of the limit, that
\begin{alignat}{2}
\label{strong:conv:nu}
&
\nut^{(n)}(k^j)|\mathbf{D}({\bu}^j)|^2
\xrightarrow[j\to\infty]{}\nut^{(n)}(k)|\mathbf{D}({\bu})|^2\ \ \mbox{in}\ \ L^1(0,T;L^1(\Omega)), && \\
\label{strong:conv:alpha}
&
|{\bu}^j|^{\alpha}
\xrightarrow[j\to\infty]{} |{\bu}|^{\alpha}\ \ \mbox{in}\ \ L^1(0,T;L^1(\Omega)). &&
\end{alignat}

With respect to the turbulent diffusion term, by estimate (\ref{est:nD:grad:kj}), we can infer the existence of $\varpi\in L^\varrho(0,T;L^\varrho(\Omega))$ such that
\begin{equation}
\label{wc:nuD:kj:grad}
\nu_D^{(n)}(k^j)\nabla k^j \xrightharpoonup[j\to\infty]{} \varpi\ \ \mbox{in}\ \ L^\varrho(0,T;L^\varrho(\Omega)),\qquad
\varrho< \varrho_2,
\end{equation}
for $\varrho_2$ defined in \eqref{est:nD:grad:kj}.
Then, we observe that from (\ref{f:Lambda}) and (\ref{Bdn:norm:Lambda}) one has
\begin{equation}\label{wc:Lambda(kj):L2}
\Lambda(k^j) \xrightharpoonup[j\to\infty]{} \Lambda(k)\ \ \mbox{in}\ \ L^2(0,T;W^{1,2}_0(\Omega)).
\end{equation}
Combining assumption (\ref{e-visc-Carath}) with
(\ref{trunc:T})-(\ref{comp:n}) and (\ref{ae:conv:kj}), one has
\begin{equation}\label{ae:nuD(kj)}
\nu_D^{(n)}(k^j)\left(1+k^j\right)^{-\frac{\zeta-\delta-1}{2}}
\xrightarrow[j\to\infty]{}\nu_D^{(n)}(k)\left(1+k\right)^{-\frac{\zeta-\delta-1}{2}}\quad \mbox{a.e.\ \ in}\ \ Q_T
\end{equation}
for $\delta>0$ so small that $\zeta-\delta-1>0$.
Now, using assumption (\ref{f:diff-turb}) and (\ref{trunc:T})-(\ref{comp:n}), together with estimate (\ref{est:kj:rho:q<rk}), we can show that
\begin{equation}\label{bd:nuD(kj):2:b}
\begin{split}
&
\int_0^T\left\|\nu_D^{(n)}(k^j(t))\left(1+k^j(t)\right)^{-\frac{\zeta-\delta-1}{2}}\right\|_2^2dt
\leq C_1\left(1+\int_0^T\left\|k^j(t)\right\|_{\zeta+\delta+1}^{\zeta+\delta+1}dt\right) \leq \\
&
C_2\left[1+\left(\int_0^T\left\|k^j(t)\right\|_{r_k}^{r_k}dt\right)^{\frac{\zeta+\delta+1}{r_k}}\right]
\leq
K_1+\frac{K_2}{\delta}\quad \forall\ \delta>0\ \ \mbox{small},
\end{split}
\end{equation}
and for some positive constants $C_1$, $C_2$, $K_1$ and $K_2$.
Note that, in view of (\ref{ru:rk}), and for $\delta>0$ sufficiently small, $\zeta+\delta+1<r_k$.
Now, the Vitali-Hahn-Saks theorem, (\ref{ae:nuD(kj)}) and (\ref{bd:nuD(kj):2:b}) imply
\begin{equation}\label{sc:nuD(kj)}
\nu_D^{(n)}(k^j)\left(1+k^j\right)^{-\frac{\zeta-\delta-1}{2}}
\xrightarrow[j\to\infty]{}\nu_D^{(n)}(k)\left(1+k\right)^{-\frac{\zeta-\delta-1}{2}}\quad \mbox{in}\ \ L^2(0,T;L^2(\Omega)).
\end{equation}
As a consequence of (\ref{f:Lambda}), (\ref{wc:Lambda(kj):L2}) and (\ref{sc:nuD(kj)}), we can justify that
\begin{equation}\label{wc:nuD(kj):infty}
\begin{split}
& \int_0^T\int_\Omega  \nu_D^{(n)}(k^j)\nabla k^j\cdot\bm{\omega}\,dxdt = \int_0^T\int_\Omega \nu_D^{(n)}(k^j)\left(1+k^j\right)^{-\frac{\zeta-\delta-1}{2}}\nabla \Lambda(k^j)\cdot\bm{\omega}\,dxdt \\
\xrightarrow[j\to\infty]{}
& \int_0^T\int_\Omega \nu_D^{(n)}(k)\left(1+k\right)^{-\frac{\zeta-\delta-1}{2}}\nabla \Lambda(k)\cdot\bm{\omega}\,dxdt=
\int_0^T\int_\Omega  \nu_D^{(n)}(k)\nabla k\cdot\bm{\omega}\,dxdt
\end{split}
\end{equation}
for all $\bm{\omega}\in C^\infty_0((0,T)\times\Omega)^d$.
Hence, by virtue of the convergence (\ref{wc:nuD(kj):infty}) we can readily see that in (\ref{wc:nuD:kj:grad}) it must be $\varpi=\nu_D^{(n)}(k)\nabla k$, i.e.
\begin{equation}
\label{wc:nuD:kj:grad:lim}
\nu_D^{(n)}(k^j)\nabla k^j \xrightharpoonup[j\to\infty]{} \nu_D^{(n)}(k)\nabla k\ \ \mbox{in}\ \ L^\varrho(0,T;L^\varrho(\Omega)),\qquad
\varrho<\varrho_2.
\end{equation}

On the other hand, we can combine (\ref{ae:conv:uj}) and (\ref{ae:conv:kj}) with (\ref{est:kj:uj}) so that
\begin{equation}
\label{wc:kjuj:varsigma}
k^{j}{\bu}^{j} \xrightharpoonup[j\to\infty]{} k{\bu}\ \ \mbox{in}\ \ L^\varrho(0,T;L^\varrho(\Omega)),\quad \varrho<\varrho_3,
\end{equation}
for $\varrho_3$ defined in \eqref{est:kj:uj:rho}.

Next, we combine assumptions (\ref{e-visc-Carath}) and (\ref{Hyp:theta:gamma}) with (\ref{trunc:T})-(\ref{comp:n}), (\ref{est:nuPn:kj:uj:i}) and (\ref{ae:conv:uj})-(\ref{ae:conv:kj}), to show that
\begin{alignat}{2}
\label{ae:conv:nuP(kj)ub:j}
&
\nu_P^{(n)}(k^{j})|{\bu}^{j}|^\beta\xrightarrow[l\to\infty]{}\nu_P^{(n)}(k)|{\bu}|^\beta\quad \mbox{a.e.\ \ in}\ \ Q_T, \\
\label{bd:nuP(kj)ujb}
&
\int_0^T\int_\Omega \left|\nu_P^{(n)}(k^{j})|{\bu}^{j}|^\beta\right|^\varrho dxdt\leq C,\qquad \varrho<\varrho_4,
\end{alignat}
for some positive constant $C$ not depending on $j$ (nor on $n$), and where $\varrho_4$ is given by \eqref{est:nuPn:kj:uj:r}.
In view of (\ref{ae:conv:nuP(kj)ub:j}) and (\ref{bd:nuP(kj)ujb}), we can use once more the Vitali-Hahn-Saks theorem so that
\begin{equation}\label{w:conv:nuP(kj)ujb}
\nu_P^{(n)}(k^{j})|{\bu}^{j}|^\beta\xrightarrow[j\to\infty]{}  \nu_P^{(n)}(k)|{\bu}|^\beta\ \ \mbox{in}\ \ L^{\varrho}(0,T;L^{\varrho}(\Omega)),\qquad \varrho<\varrho_4.
\end{equation}

Regarding the turbulent dissipation term, we can deduce from (\ref{e-visc-Carath}) and (\ref{ae:conv:kj}) that
\begin{equation}\label{ae:conv:vare(kj)}
\varepsilon(k^j)\xrightarrow[j\to\infty]{}\varepsilon(k)\quad \mbox{a.e.\ \ in}\ \ Q_T.
\end{equation}
And by using assumptions (\ref{f:dissip-turb}) and (\ref{eq:Cond2}), together with estimate (\ref{est:kj:rho:q<rk}), we can prove that
\begin{equation}\label{bd:vare(kj):q}
\int_0^T\int_\Omega \big|\varepsilon(k^j)\big|^{\varrho}dxdt\leq
C\int_0^T\int_\Omega \big|k^j\big|^{\varrho(\vartheta+1)}dxdt\leq \frac{K}{\delta}\quad \forall\ \delta>0\ \ \mbox{small},\quad \varrho<\varrho_5:=\frac{d\zeta +d+2}{d(\vartheta+1)},
\end{equation}
for some positive constants $C$ and $K$.
Using again the Vitali-Hahn-Saks theorem, we can see that (\ref{ae:conv:vare(kj)}) and (\ref{bd:vare(kj):q}) imply
\begin{equation}\label{sc:conv:vare(kj)}
\varepsilon(k^j)\xrightarrow[j\to\infty]{}\varepsilon(k)\quad \mbox{in}\ \ L^{\varrho}(0,T;L^{\varrho}(\Omega)),\qquad \varrho<\varrho_5.
\end{equation}
Note that, due to assumption (\ref{eq:Cond2}), $\varrho_5>1$.

Now, using the convergence results (\ref{w:conv:kj:W1q}), (\ref{w:conv:kj'}), (\ref{str:conv:uj}), (\ref{strong:conv:nu}), (\ref{wc:nuD:kj:grad:lim}), (\ref{wc:kjuj:varsigma}), (\ref{sc:conv:vare(kj)}) and (\ref{w:conv:nuP(kj)ujb}), we can pass (\ref{weak-form-k:tr:j:w}) to the limit $j\to\infty$ so that (\ref{weak-form-k:n}) holds true for any $w\in W^{1,\infty}_0(\Omega)$.

Moreover, reasoning as we did for \eqref{eq:0:uj:kj}, we also can show that
\begin{equation}\label{eq:0:un:kn}
\bu(0)=\bu_{n,0}\quad\mbox{and}\quad  k(0)=k_{n,0}\quad\mbox{a.e. in}\ \Omega.
\end{equation}
The proof of Proposition~\ref{prop:exist:n} is thus concluded.
\end{proof}

In the final two sections we will conclude the proof Theorem~\ref{thm:exist}.
Note that, right at the beginning of the proof of Proposition~\ref{prop:exist:n}, we discarded the subscript $n$, which will now be recovered so that we can proceed.

\section{Passing to the limit as $n\to\infty$}\label{Sect-Est-ind(n)}  

\begin{proof}(Concluding the proof of Theorem~\ref{thm:exist})
From Proposition~\ref{prop:exist:n}, we know that for each $n\in\mathds{N}$ there exists a couple $({\bu}_n,k_n)$ of functions such that
\eqref{weak-form-u:n} and \eqref{weak-form-k:n} are satisfied.
Using continuity arguments, integration in-time of \eqref{weak-form-u:n} and \eqref{weak-form-k:n}, and \eqref{eq:0:un:kn}, we can see that
\begin{equation}
\begin{split}\label{weak-form-un}
& -\int_{0}^{T}\int_{\Omega} \bu_n\cdot\partial_t\bm{\varphi}\,dxdt - \int_{0}^{T}\int_{\Omega}\Phi_n(|\bu_n|^2)\bu_n\otimes\bu_n:\nabla\bm{\varphi}\,dxdt \\
&  +\int_{0}^{T}\int_{\Omega}\nut^{(n)}(k_n)\,\mathbf{D}(\bu_n):\nabla\bm{\varphi}\,dxdt
   +\int_{0}^{T}\int_{\Omega}\left(c_{Da}+c_{Fo}|\bu_n|^{\alpha-2}\right)\bu_n\cdot\bm{\varphi}\,dxdt = \\
&
\int_\Omega\bu_{n,0}\cdot\bm{\varphi}(0)\,dx + \int_{0}^{T}\int_{\Omega}\bg\cdot\bm{\varphi}\,dxdt
\end{split}
\end{equation}
and
\begin{equation}\label{weak-form-kn}
\begin{split}
& -\int_{0}^{T}\int_{\Omega} k_n\partial_t\omega\,dxdt - \int_{0}^{T}\int_{\Omega}k_n\bu_n\cdot\nabla\omega\,dxdt
  + \int_{0}^{T}\int_{\Omega}\nu_D^{(n)}(k_n)\nabla k_n\cdot\nabla\omega\,dxdt  \\
& + \int_{0}^{T}\int_{\Omega}\varepsilon(k_n)\omega\,dxdt = \\
& \int_\Omega k_{n,0}\omega(0)\,dx + \int_{0}^{T}\int_{\Omega}\nut^{(n)}(k_n)|\mathbf{D}(\bu_n)|^2\omega\,dxdt + \int_{0}^{T}\int_{\Omega}\nu_P^{(n)}(k_n)|\bu_n|^\beta \omega\,dxdt
\end{split}
\end{equation}
are verified for all
$\bm{\varphi}\in C^\infty(Q_T)^d$, with $\operatorname{div}\bm{\varphi}=0$ in $Q_T$ and $\operatorname{supp}\bm{\varphi}\subset\subset\Omega\times[0,T)$, and for all $\omega\in C^\infty(Q_T)$, with $\omega\geq 0$ in $Q_T$ and $\operatorname{supp}\omega\subset\subset\Omega\times[0,T)$.

Using \eqref{weak-form-u:n} and \eqref{weak-form-k:n}, and proceeding as we did in the previous sections, we can show the estimates \eqref{est2:eq:u:T}-\eqref{est4:eq-u:jl}, \eqref{est:kj:infty:sup:1}, (\ref{est:grad:kj:L2:0})-(\ref{est:grad:kj:Lq:0}), \eqref{est:kj:rho:q<rk}, \eqref{est:nD:grad:kj}, \eqref{est:kj:uj}, \eqref{est:nuPn:kj:uj:i} and \eqref{est:indep:jn:3:0} hold for $\bu_n$ and $k_n$.
As a consequence, and in view of the Banach-Alaoglu theorem, we have for some subsequences
\begin{alignat}{2}
 & \label{w*:conv:un}
\bu_n\xrightharpoonup[n\to\infty]{\ast} \bu\ \ \mbox{in}\ \ L^\infty(0,T;\mathbf{H}), && \\
  & \label{w:conv:un}
\bu_n\xrightharpoonup[n\to\infty]{} \bu,\ \ \mbox{in}\ L^2(0,T;\mathbf{V})\cap L^{r_u}(0,T;L^{r_u}(\Omega)^d),\quad
\mbox{for $r_u$ given in (\ref{ru:rk})}, && \\
& 
k_n\xrightharpoonup[n\to\infty]{\ast} k\ \ \mbox{in}\ \ L^\infty(0,T;\mathcal{M}(\Omega)), && \nonumber \\
  & \label{w:conv:kn:W1q}
k_n\xrightharpoonup[n\to\infty]{} k\ \ \mbox{in}\ \ L^q(0,T;W^{1,q}_0(\Omega)),\qquad
1< q<\min\left\{2,1+\frac{d\zeta + 1}{d+1}\right\}, && \\
 & \label{w:conv:kn:rho:theta}
k_n\xrightharpoonup[n\to\infty]{} k\ \ \mbox{in}\ \ L^r(0,T;L^r(\Omega))\cap L^{\vartheta+1}(0,T;L^{\vartheta+1}(\Omega)),\quad 1\leq r<r_k,\quad
\mbox{for $r_k$ given in (\ref{ru:rk})}, \\
  & \label{w:conv:kn'}
\partial_tk_n\xrightharpoonup[n\to\infty]{} \partial_tk\ \ \mbox{in}\ \ \mathcal{M}(0,T;W^{-1,\varrho}(\Omega)),\quad 1<\varrho<\varrho_0,\quad \mbox{for $\varrho_0$ given in \eqref{varrho:0}}.
\end{alignat}
From \eqref{w:conv:kn'} one immediately has (9) of Theorem~\ref{thm:exist}.

On the other hand, using the Hölder inequality and assumption (\ref{f:visc-turb}), together with (\ref{trunc:T})-(\ref{comp:n}), and with the counterparts of (\ref{est2:eq:u:T}), (\ref{est:kj:infty:sup:1}) and (\ref{est:kj:rho:q<rk}), one has
    \begin{equation}\label{eq:est:Dun:q}
    \begin{split}
    & \int_0^T\left\|\nut^{(n)}(k_n(t))\,\mathbf{D}({\bu}_n(t))\right\|_{\varsigma}^{\varsigma}dt\leq \\
    &
    \left(\int_0^T\left\|\sqrt{\nut^{(n)}(k_n(t))}\mathbf{D}({\bu}_n(t))\right\|_{2}^{2}dt\right)^{\frac{1}{2}}
    \left(\int_0^T\left\|\sqrt{\nut^{(n)}(k_n(t))}\right\|_{\frac{2r}{\eta}}^{\frac{2r}{\eta}}dt\right)^{\frac{\eta}{2r}}\leq \\
    &  C_1\left(1+\int_0^T\left\|k_n(t)\right\|_{r}^{r}dt\right)^{\frac{\eta}{2r}}\leq
    \left\{
    \begin{array}{ll}
      C, & \mbox{if}\ r=\vartheta+1,\ \ \mbox{or} \\
      K_1 + \frac{K_2}{\delta}\quad \forall\ \delta>0\ \ \mbox{small}, &  \mbox{if}\ r<r_k,
    \end{array}
    \right.
    \end{split}
    \end{equation}
    for some positive constants $C_1$, $C$, $K_1$ and $K_2$ not depending on $n$, where $r_k$ is given in (\ref{ru:rk}) and
    \begin{equation}\label{eq:est:Dun:qk}
    \frac{1}{2}+\frac{\eta}{2r}=\frac{1}{\varsigma} \Leftrightarrow \varsigma=2-\frac{2\eta}{r+\eta}< \varsigma_1:=
    \left\{
    \begin{array}{ll}
    2-\frac{2\eta}{\vartheta+1+\eta}, & \mbox{if}\ r=\vartheta+1,\ \ \mbox{or} \\
    2-\frac{2d\eta}{d\zeta+d\eta+d+2}, & \mbox{if}\ r<r_k.
    \end{array}
    \right.
    \end{equation}
    Note that assumptions (\ref{eq:Cond1}) and (\ref{eq:Cond2}) assure that $\varsigma_1>1$ in any case.

\begin{remark}
For the exponent $\varsigma_1$ set in (\ref{eq:est:Dun:qk}), we have in the dimensions of physics interest
\begin{equation*}
\varsigma_1:=
\left\{
\begin{array}{ll}
\max\left\{\frac{2(\vartheta+1)}{\vartheta+1+\eta},\frac{2(3\zeta+5)}{3\zeta+3\eta+5}\right\}, & \mbox{if}\ d=3, \\
\max\left\{\frac{2(\vartheta+1)}{\vartheta+1+\eta},\frac{2(\zeta+2)}{\zeta+\eta+2}\right\}, & \mbox{if}\ d=2.
\end{array}
\right.
\end{equation*}
In the particular case of $d=3$, $\varsigma_1=\frac{2(3\zeta+5)}{3\zeta+3\eta+5}$ if and only if $\vartheta<\zeta+\frac{2}{3}$, which was one of the main assumptions of \cite{BLM:2011}.
In our case, would be
$\varsigma_1=\frac{2(d\eta+d+2)}{d\zeta+d\eta+d+2}$ and $\vartheta<\zeta+\frac{2}{d}$, which is assured by assumption (\ref{eq:Cond2}).
However, and contrary to (\ref{bd:vare(kj):q}), this hypothesis is used here more to simplify the presentation than a real need for our analysis.
\end{remark}

From \eqref{weak-form-u:n} we can infer that for all $t\in(0,T)$
\begin{equation}\label{eq:dt:u:n}
\begin{split}
\partial_t\bu_n(t)= &
-\mathbf{div}\big(\Phi_n(|\bu_n(t)|^2)\bu_n(t)\otimes\bu_n(t)\big) +
\mathbf{div}\big(\nut(k_n(t))\,\mathbf{D}(\bu_n(t))\big) \\
&
-\left(c_{Da}+c_{Fo}|\bu_n(t)|^{\alpha-2}\right)\bu_n(t) + \mathbf{g}(t)
\end{split}
\end{equation}
holds in the distribution sense on $\mathbf{Y}'$,
where $\mathbf{Y}'$ denotes the dual space of
\begin{equation*}
\mathbf{Y}:=\mathbf{V}\cap L^\alpha(\Omega)^d.
\end{equation*}
From \eqref{eq:est:Dun:q}, one immediately has
    \begin{equation}\label{eq:1:p:n}
    \int_0^T\left\|\mathbf{div}\Big(\nut^{(n)}(k_n(t))\,\mathbf{D}({\bu}_n(t))\Big)\right\|_{W^{-1,\varsigma}(\Omega)^d}^{\varsigma}dt\leq C,\qquad
    1<\varsigma<\varsigma_1,
    \end{equation}
for $\varsigma_1$ defined in \eqref{eq:est:Dun:qk}.

By using \eqref{Phi}-\eqref{Phi:d}, the Hölder inequality and the counterpart of \eqref{est4:eq-u:jl}, we can show that
\begin{equation}\label{est:uxu:uj'}
\int_0^T\left\|\mathbf{div}\left(\Phi_n(|\bu_n(t)|^2)\bu_n(t)\otimes\bu_n(t)\right)\right\|_{W^{-1,\varsigma}(\Omega)^d}^{\varsigma}dt\leq
C,\quad 1<\varsigma\leq\varsigma_2:=
\left\{
\begin{array}{ll}
1+\frac{2}{d}, & \mbox{if}\ r_u=\frac{2(d+2)}{d}, \\
\frac{\alpha}{2}, & \mbox{if}\ r_u=\alpha,
\end{array}
\right.
\end{equation}
for some positive constant $C$.
Note that if $r_u=\alpha$, then $\alpha>2$ and consequently $\varsigma_2>1$ in either cases.

By the same reasoning,
\begin{equation}\label{eq:3:p:n}
\begin{split}
&
\int_0^T\left\|{\bu}_n(t)\right\|_{W^{-1,\varsigma_3}(\Omega)^d}^{\varsigma_3}dt
+
\int_0^T\left\||{\bu}_n(t)|^{\alpha-2}{\bu}_n(t)\right\|_{W^{-1,\varsigma_4}(\Omega)^d}^{\varsigma_4}dt\leq C,\quad
\varsigma_3:=r_u,\quad \varsigma_4:=\frac{r_u}{\alpha-1},
\end{split}
\end{equation}
for some positive constant $C$.

Finally, by the Hölder inequality and assumption (\ref{g:V'}), we have
\begin{equation}\label{eq:4:p:n}
\int_0^T\left\|\mathbf{g}(t)\right\|_{W^{-1,\varsigma_5}(\Omega)^d}^{\varsigma_5}dt\leq C,\qquad \varsigma_5:=2.
\end{equation}

Now, using \eqref{eq:dt:u:n} and (\ref{eq:1:p:n})-\eqref{eq:4:p:n}, one has
\begin{equation}\label{est:dt:u:n}
\int_0^T\left\|\partial_t\bu_n(t)\right\|_{W^{-1,\varsigma}(\Omega)^d}^{\varsigma}dt \leq C,\quad
1<\varsigma<\varsigma_0:=\min\left\{\varsigma_1,\varsigma_2,\varsigma_3,\varsigma_4,\varsigma_5\right\}.
\end{equation}
Attending to \eqref{ru:rk} and assumption \eqref{eq:Cond2}, the expression for $\varsigma_0$ simplifies as follows,
\begin{equation}\label{varsigma:0}
\varsigma_0:=\min\left\{
\frac{2\left(d\zeta+d+2\right)}{d\zeta+d\eta+d+2},
\max\left\{1+\frac{2}{d},\frac{\alpha}{2}\right\},
\frac{\max\left\{\frac{2(d+2)}{d},\alpha\right\}}{\alpha-1}\right\}
\end{equation}
Observe that $\varsigma_1,\ \varsigma_2,\ \varsigma_3,\ \varsigma_4,\ \varsigma_5>1$ and thus $\varsigma:1<\varsigma<\varsigma_0$ can be chosen.

Combining \eqref{est:dt:u:n} with the Banach-Alaoglu theorem, we also have for some subsequence
\begin{equation}\label{w:conv:un'}
\partial_t\bu_n\xrightharpoonup[n\to\infty]{} \partial_t\bu\ \ \mbox{in}\ \ L^{\varsigma}(0,T;W^{-1,\varsigma}(\Omega)^d),\qquad
1<\varsigma<\varsigma_0,
\end{equation}
which proves (8) of Theorem~\ref{thm:exist}.

Taking into account \eqref{est:indep:jn:3:0} (with $k_n$ in the place of $k^j$) and \eqref{est:dt:u:n}, we can justify, arguing similarly as we did for obtaining  \eqref{str:conv:uj}-\eqref{str:conv:kj} and
\eqref{ae:conv:uj}-\eqref{ae:conv:kj}, the existence of subsequences such that
\begin{alignat}{2}
& \nonumber
\bu_n\xrightarrow[n\to\infty]{}\bu\ \ \mbox{in}\ \ L^q(0,T;L^q(\Omega)^d),\qquad 1\leq q< r_u, \\
& \nonumber
k_n\xrightarrow[n\to\infty]{}k\ \ \mbox{in}\ \ L^q(0,T;L^q(\Omega)),\qquad 1\leq q<r_k, \\
& \label{ae:conv:un}
\bu_n\xrightarrow[n\to\infty]{}\bu\quad \mbox{a.e.\ \ in}\ \ Q_T,  && \\
& \label{ae:conv:kn}
k_n\xrightarrow[n\to\infty]{}k\quad \mbox{a.e.\ \ in}\ \ Q_T. &&
\end{alignat}
And, similarly to (\ref{w:conv:ujxuj})-(\ref{w:conv:nuT(k)Dv:j:sqrt}), we can use the convergence results \eqref{w:conv:un}, \eqref{ae:conv:un} and \eqref{ae:conv:kn}, to show that
\begin{alignat}{2}
& \label{w:conv:unxun}
\Phi_n(|u_n|^2)\bu_n\otimes\bu_n\xrightharpoonup[n\to\infty]{} \bu\otimes\bu\ \ \mbox{in}\ \ L^{\frac{d+2}{d}}(0,T;L^{\frac{d+2}{d}}(\Omega)^{d\times d}), && \\
& \label{w:conv:f(un)}
|\bu_n|^{\alpha-2}\bu_n\xrightharpoonup[n\to\infty]{} |\bu|^{\alpha-2}\bu
\ \ \mbox{in}\ \ L^{\alpha'}(0,T;L^{\alpha'}(\Omega)^d), && \\
& \label{w:conv:nuT(k)Dv:n}
\nut^{(n)}(k_n)\mathbf{D}(\bu_n)\xrightharpoonup[n\to\infty]{}  \nut(k)\mathbf{D}(\bu)\ \ \mbox{in}\ \ L^2(0,T;L^{2}(\Omega)^{d\times d}), \\
& \label{w:conv:nuT(k)Dv:n:sqrt}
\sqrt{\nut^{(n)}(k_n)}\mathbf{D}(\bu_n)\xrightharpoonup[n\to\infty]{}  \sqrt{\nut(k)}\mathbf{D}(\bu)\ \ \mbox{in}\ \ L^2(0,T;L^{2}(\Omega)^{d\times d}). &&
\end{alignat}
Note that in the convergence result \eqref{w:conv:unxun}, we also have used the definition of the function $\Phi$ given at \eqref{Phi}-\eqref{Phi:d}.

Then, passing the equation \eqref{weak-form-un} to the limit $n\to\infty$, using for that purpose the convergence results \eqref{conv:un:u0}, \eqref{w:conv:un}, and \eqref{w:conv:unxun}-\eqref{w:conv:nuT(k)Dv:n}, we prove the validity of \eqref{weak-form-u}.

On the other hand, arguing exactly as we did for \eqref{wc:nuD:kj:grad:lim}, \eqref{wc:kjuj:varsigma}, \eqref{w:conv:nuP(kj)ujb}, and \eqref{sc:conv:vare(kj)}, using in this case \eqref{ae:conv:un}-\eqref{ae:conv:kn} instead, we can show that
\begin{alignat}{2}
& \label{wc:nuD:kn:grad:lim}
\nu_D^{(n)}(k_n)\nabla k_n \xrightharpoonup[n\to\infty]{} \nu_D(k)\nabla k\ \ \mbox{in}\ \ L^\varrho(0,T;L^\varrho(\Omega)),\qquad
1<\varrho<\varrho_2, && \\
& \label{wc:knun:varsigma}
k_n\bu_n \xrightharpoonup[n\to\infty]{} k\bu\ \ \mbox{in}\ \ L^\varrho(0,T;L^\varrho(\Omega)),\qquad
1<\varrho<\varrho_3, && \\
& \label{w:conv:nuP(kn)ujb}
\nu_P^{(n)}(k_n)|\bu_n|^\beta\xrightarrow[n\to\infty]{}  \nu_P(k)|\bu|^\beta\ \ \mbox{in}\ \ L^\varrho(0,T;L^\varrho(\Omega)),\qquad 1<\varrho<\varrho_4, && \\
& \label{sc:conv:vare(kn)}
\varepsilon(k_n)\xrightarrow[n\to\infty]{}\varepsilon(k)\quad \mbox{in}\ \ L^\varrho(0,T;L^\varrho(\Omega)),\qquad 1<\varrho<\varrho_5, &&
\end{alignat}
where $\varrho_2$, $\varrho_3$, $\varrho_4$ and $\varrho_5$ are defined in \eqref{est:nD:grad:kj}, \eqref{est:kj:uj:rho}, \eqref{est:nuPn:kj:uj:r} and \eqref{bd:vare(kj):q}, respectively.

Moreover, \eqref{w:conv:nuT(k)Dv:n:sqrt} and the weak lower semicontinuity of the norm imply
\begin{equation}\label{wlsc:L2norm:n}
\int_{0}^{T}\int_\Omega \nut(k)|\bD(\bu)|^2\omega\,dxdt \leq
\liminf_{n\to\infty}\int_{0}^{T}\int_\Omega \nut^{(n)}(k_n)|\bD(\bu_n)|^2\omega\,dxdt.
\end{equation}
Finally, using the convergence results \eqref{sc:rn00}, \eqref{w:conv:kn:rho:theta}, \eqref{wc:nuD:kn:grad:lim}-\eqref{sc:conv:vare(kn)} and \eqref{wlsc:L2norm:n}, we can pass \eqref{weak-form-kn} to the limit $n\to\infty$ and we obtain \eqref{weak-form-k}.

\section{Attainment of the initial conditions}\label{Sect-Att-ic}

Similarly to (\ref{eq:0:un:kn})$_1$, we can use \eqref{weak-form-u:n} and invoke (\ref{w*:conv:un})-(\ref{w:conv:un}), (\ref{w:conv:un'}) and (\ref{w:conv:nuT(k)Dv:n:sqrt}) to prove that $\bu\in L^\infty(0,T;\mathbf{H})$ is weakly continuous with values in $\mathbf{H}$, and $\bu(0)=\bu_0$ in the sense of \eqref{attain:ic:u0k0}.

Repeating the same arguments used to show  \eqref{kj>=0} and \eqref{epsilon(kj)>=0}, we can prove that
\begin{equation}\label{e:epsilon(kn)>=0}
k_n\geq 0,\quad
\varepsilon(k_n)\geq 0\quad\mbox{a.e. in}\ Q_T.
\end{equation}
To prove that
\begin{equation}\label{eq:k0:k(0)}
k(0)=k_{0},
\end{equation}
we start by integrating (\ref{weak-form-k:n}) between $0$ and $t\in(0,T)$, next arguing as we did for proving (\ref{est2:eq-k:a:sup}), and then using (\ref{eq:0:un:kn})$_2$ and \eqref{e:epsilon(kn)>=0}, we arrive at
\begin{equation*}
\|\mathcal{H}_1(k_n(t))\|_1  \leq
\|\mathcal{H}_1(k_{n,0})\|_1 +
\int_0^t\int_{\Omega}\nut^{(n)}(k_n)|\mathbf{D}(\bu_n)|^2dxd\tau + \int_0^t\int_{\Omega}\nu_P^{(n)}(k_n)|\bu_n|^\beta dxd\tau,
\end{equation*}
where $\mathcal{H}_1$ is the function defined in (\ref{funct:H1}).
Combining the Fatou lemma with (\ref{sc:rn00}), (\ref{ae:conv:kn}), (\ref{w:conv:nuP(kn)ujb}) and (\ref{wlsc:L2norm:n}), we obtain
\begin{equation*}
\|\mathcal{H}_1(k(t))\|_1  \leq
\|\mathcal{H}_1(k_0)\|_1 +
\int_0^t\int_{\Omega}\nut(k)|\mathbf{D}(\bu)|^2dxd\tau + \int_0^t\int_{\Omega}\nu_P(k)|\bu|^\beta dxd\tau.
\end{equation*}
Hence,
\begin{equation}\label{limit:sup:t0+}
\limsup_{t\to 0^+}\|\mathcal{H}_1(k(t))\|_1  \leq
\|\mathcal{H}_1(k_0)\|_1.
\end{equation}
We now test (\ref{weak-form-k:tr:j:w}) with
\begin{equation}\label{test:f:om}
\omega=\mathcal{T}_1(k_n)\mathcal{H}_1(k_n)^{-\frac{1}{2}}\phi, \quad \phi\in C^\infty_0(\Omega),\quad \phi\geq 0\ \ \mbox{a.e. in}\ \ \Omega,
\end{equation}
where $\mathcal{T}_1(k_n)$ is the truncation of $k_n$ defined in (\ref{trunc:T}) for $n=1$ and $\mathcal{H}_1$ is the primitive function of $\mathcal{T}_1$ defined in (\ref{funct:H1}).
Note that, since $\phi\in C^\infty_0(\Omega)$ and
\begin{equation*}
\mathcal{T}_1(k_n)\mathcal{H}_1(k_n)^{-\frac{1}{2}}=
\left\{
\begin{array}{cc}
\sqrt{2}, & k_n<1, \\
0, & k_n\geq 1,
\end{array}
\right.
\end{equation*}
the function $\omega$ given by (\ref{test:f:om}) is in fact an admissible test function.
Integrating the resulting equation between $0$ and $t$, and proceeding as in~\cite{O:2024a}, we obtain
\begin{equation*}
\label{leq-k_n:T1:H1}
\begin{split}
& \int_{\Omega}\mathcal{H}_1(k_n(t))^{\frac{1}{2}}\phi\,dx -
  \int_0^t\int_{\Omega}\mathcal{H}_1(k_n)^{\frac{1}{2}}\bu_n\cdot\nabla \phi\,dxd\tau + \\
&
\frac{1}{2}\int_0^t\int_{\Omega}\nu_D^{(n)}(k_n)\mathcal{T}_1(k_n)\mathcal{H}_1(k_n)^{-\frac{1}{2}}\nabla k_n\cdot\nabla\phi\,dxd\tau
+ \frac{1}{2}\int_0^t\int_{\Omega}\varepsilon(k_n)\mathcal{T}_1(k_n)\mathcal{H}_1(k_n)^{-\frac{1}{2}}\phi\,dxd\tau
\\
\geq &
\int_{\Omega}\mathcal{H}_1(k_n(0))^{\frac{1}{2}}\phi\,dx.
\end{split}
\end{equation*}
Using (\ref{eq:0:un:kn})$_2$, together with (\ref{sc:rn00}), (\ref{w:conv:kn:rho:theta}), (\ref{wc:nuD:kn:grad:lim}), (\ref{wc:knun:varsigma}) and (\ref{sc:conv:vare(kn)}), we get
\begin{equation*}
\begin{split}
& \int_{\Omega}\mathcal{H}_1(k(t))^{\frac{1}{2}}\phi\,dx -
  \int_0^t\int_{\Omega}\mathcal{H}_1(k)^{\frac{1}{2}}{\bu}\cdot\nabla \phi\,dxd\tau + \\
&
\frac{1}{2}\int_0^t\int_{\Omega}\nu_D^{(n)}(k)\mathcal{T}_1(k)\mathcal{H}_1(k)^{-\frac{1}{2}}\nabla k\cdot\nabla\phi\,dxd\tau
+ \frac{1}{2}\int_0^t\int_{\Omega}\varepsilon(k)\mathcal{T}_1(k)\mathcal{H}_1(k)^{-\frac{1}{2}}\phi\,dxd\tau
\\
\geq &
\int_{\Omega}\mathcal{H}_1(k_{n,0})^{\frac{1}{2}}\phi\,dx
\end{split}
\end{equation*}
for a.e. $t\in(0,T)$.
Taking the $\liminf$, as $t\to 0^+$, and using a density argument, one has
\begin{equation}\label{leq:lim:inf}
\begin{split}
& \liminf_{t\to 0^+}\int_{\Omega}\mathcal{H}_1(k(t))^{\frac{1}{2}}\phi\,dx \geq
\int_{\Omega}\mathcal{H}_1(k_0)^{\frac{1}{2}}\phi\,dx\quad \forall\ \phi\in\ L^2(\Omega),
\end{split}
\end{equation}
with $\phi\geq 0$ a.e. in $\Omega$.
Now, we can use (\ref{limit:sup:t0+}) and (\ref{leq:lim:inf}), together with the properties of $\limsup$ and $\liminf$, and with (\ref{e:epsilon(kn)>=0})$_1$ and (\ref{sc:rn00}), to prove that
\begin{equation}\label{leq:lim:inf:sup}
\begin{split}
& \lim_{t\to 0^+}\left\|\mathcal{H}_1(k(t))^{\frac{1}{2}}-\mathcal{H}_1(k_0)^{\frac{1}{2}}\right\|_2^2
= \\
& \lim_{t\to 0^+}\left(\left\|\mathcal{H}_1(k(t))\right\|_1 + \left\|\mathcal{H}_1(k_0)\right\|_1 -
2\int_{\Omega}\mathcal{H}_1(k(t))^{\frac{1}{2}}\mathcal{H}_1(k_0)^{\frac{1}{2}}dx\right) \leq \\
& \limsup_{t\to 0^+}\left\|\mathcal{H}_1(k(t))\right\|_1 + \left\|\mathcal{H}_1(k_0)\right\|_1
-2\liminf_{t\to 0^+}\int_{\Omega}\mathcal{H}_1(k(t))^{\frac{1}{2}}\mathcal{H}_1(k_0)^{\frac{1}{2}}dx \leq \\
 &
\left\|\mathcal{H}_1(k_0)\right\|_1 + \left\|\mathcal{H}_1(k_0)\right\|_1
-2\int_{\Omega}\mathcal{H}_1(k_0)dx=0.
\end{split}
\end{equation}
As a consequence of (\ref{leq:lim:inf:sup}), we achieve to $k(0)=k_0$ in the sense of \eqref{attain:ic:u0k0},
which concludes the proof of Theorem~\ref{thm:exist}.
\end{proof}

\subsection*{Acknowledgments}
The author was partially supported by the Portuguese Foundation for Science and Technology, Portugal, under the project: UIDB/04561/2020.

\end{document}